\documentclass[a4paper,12pt]{article}

\usepackage{amsmath}
\usepackage{amsthm}
\usepackage{enumerate}
\usepackage{amssymb,amsfonts,latexsym,mathtools}
\usepackage{makecell}
\usepackage{mathrsfs}
\usepackage{bm}
\usepackage{cases}
\usepackage{color}

\newcommand{\Levy}{L\'{e}vy}

\newcommand{\F}{\mathscr{F}}

\newcommand{\Z}{\mathbb{Z}}
\newcommand{\Q}{\mathbb{Q}}

\renewcommand{\P}{\mathbb{P}}
\usepackage{multirow}
\newcommand{\E}{\mathbb{E}}
\usepackage{url}
\usepackage{tikz}
\usetikzlibrary{
  arrows.meta,
  calc,
  decorations.pathmorphing,
  positioning
}

\numberwithin{equation}{section}

\makeatletter
\renewcommand\section{\@startsection {section}{1}{\z@}%
{-3.5ex \@plus -1ex \@minus -.2ex}%
{2.3ex \@plus.2ex}%
{\normalfont\large\bf}}
\makeatother

\makeatletter
\renewcommand\subsection{\@startsection {subsection}{1}{\z@}%
{-3.5ex \@plus -1ex \@minus -.2ex}%
{2.3ex \@plus.2ex}%
{\normalfont\normalsize\bf}}
\makeatother

\theoremstyle{plain}
\newtheorem{thm}{Theorem}[section]
\newtheorem{lem}[thm]{Lemma}
\newtheorem{cor}[thm]{Corollary}
\newtheorem{prop}[thm]{Proposition}
\theoremstyle{definition}

\allowdisplaybreaks

\usepackage{hyperref}
\hypersetup{
setpagesize=false,
 bookmarksnumbered=true,
 bookmarksopen=true,
 colorlinks=true,
 linkcolor=blue,
 citecolor=red,
}

\begin{document}
\begin{center}
\Large \textbf{Penalizations of Two-Dimensional Simple Random Walk by Direction-Dependent Excursion Counts and Edge-Crossing Counts}
\end{center}
\begin{center}
Syota Esaki\footnote{
\begin{tabular}[t]{@{}l@{}}
Affiliation: Faculty of Science and Technology, Tokyo University of Science.\\
E-mail: \url{sesaki@rs.tus.ac.jp}
\end{tabular}
} and Kohki Iba\footnote{
\begin{tabular}[t]{@{}l@{}}
Affiliation: Graduate School of Science, The University of Osaka.\\
E-mail: \url{kohki.iba@gmail.com}
\end{tabular}
}
\end{center}
\begin{abstract}
We study penalizations of the two-dimensional {simple random walk} by functions of direction-dependent excursion counts and edge-crossing counts. The former record completed excursions with prescribed pairs of departure and return directions at the origin, whereas the latter record crossings of a prescribed set of unoriented edges incident to the origin. We determine the distributions of their almost surely finite terminal values and derive explicit path decomposition formulas involving the random walk conditioned to avoid the origin. Moreover, both families of penalized measures can be represented by weighting and normalizing the same universal $\sigma$-finite measure associated with local-time penalization at the origin. Thus, direction-dependent excursion counts, edge-crossing counts, and local time at the origin give rise to penalizations belonging to a common universality class.
\end{abstract}



\section{Introduction}
\subsection{Penalization Problem}
The \emph{penalization problem} for a stochastic process concerns reweighting the original probability measure by a nonnegative path-dependent functional and studying both the limiting measure obtained as the observation time tends to infinity and the behavior of the process under this limiting measure.

More precisely, let $X=(X_t)_{t\in\mathbb{T}}$ be a Markov process with state space $\E$ and time set $\mathbb{T}$, and let $\mathbb{F}=(\F_t)_{t\in\mathbb{T}}$ denote its natural filtration. Whenever necessary, we assume that $\mathbb{F}$ is right-continuous. For $x\in\E$, let $\P_x$ denote the law of $X$ started from $x$, and let $\P_x[\cdot]$ denote expectation with respect to this measure. Given a nonnegative $(\F_t)$-adapted process $\Gamma=(\Gamma_t)_{t\in\mathbb{T}}$, the penalization problem with weight $\Gamma$ consists of studying the limit
\begin{align}
\label{penalization-limit}
\lim_{t\to \infty}\frac{\P_x[1_{\Lambda_s}\cdot \Gamma_t]}{\P_x[\Gamma_t]}
\qquad \text{for } s\in \mathbb{T},\ \Lambda_s\in \F_s.
\end{align}
When this limit is represented by a probability measure on the path space, we call this measure the \emph{penalized measure} and the process under this measure the \emph{penalized process}. Depending on the choice of the weight $\Gamma$, the resulting penalized process admits various probabilistic interpretations.

For example, suppose that the local time $L^z=(L_t^z)_{t\in\mathbb{T}}$ at some point $z\in\E$ is well defined. One may then consider the weight $\Gamma_t=e^{-L_t^z}.$ Since this weight decreases as the local time increases, paths accumulating a large amount of local time at $z$ are assigned smaller weights. Thus, heuristically, the corresponding penalized process can be interpreted as a process whose visits to $z$ are suppressed. As another example, when $X$ is a real-valued process, one may consider the weight $\Gamma_t=e^{-\overline{X}_t},\ \overline{X}_t:=\sup_{u\le t}X_u.$ In this case, paths attaining a large running supremum are assigned smaller weights. Hence, heuristically, the corresponding penalized process can be interpreted as a process for which reaching high levels is suppressed.

\emph{Conditioning problems} can be viewed as a special class of penalization problems. Indeed, let $T_A$ be the hitting time of a set $A\subset\E$, and set $\Gamma_t=1_{\{t<T_A\}}.$ Then the ratio appearing in (\ref{penalization-limit}) is given by $\P_x(\Lambda_s\mid t<T_A)$. Consequently, the process obtained by taking the limit as $t\to\infty$ may be interpreted as the process conditioned to avoid the set $A$.

Penalization problems for one-dimensional Brownian motion have been extensively studied for various weights involving, among other quantities, local times and the running supremum; see Roynette--Vallois--Yor \cite{Roynette-Vallois-Yor-2,Roynette-Vallois-Yor-1,RVY-excursion,RVY-Last}, Najnudel \cite{Najnudel}, and Roynette--Yor \cite{Roynette-Yor-tenpenal,Roynette-Yor-Additive}. Subsequently, the study of penalization problems was extended to a wide variety of stochastic processes. Roynette--Vallois--Yor \cite{Roynette-Vallois-Yor-Bessel} studied Bessel processes; Salminen--Vallois \cite{Salminen-Vallois}, Profeta \cite{Profeta}, and Profeta--Yano--Yano \cite{Profeta-Yano-Yano}, one-dimensional diffusion processes; Yano--Yano--Yor \cite{Yano-Yano-Yor, Yano-Yano-Yor-sup}, one-dimensional stable processes; Abraham--Debs \cite{Abraham-Debs} and Abraham--Boulal--Debs \cite{Abraham-Boulal-Debs}, Galton--Watson processes; Takeda--Yano \cite{Takeda-Yano}, Iba--Yano \cite{Iba-Yano}, Iba \cite{Iba3,Iba-Kackilling}, and Takeda \cite{Takeda-sup}, one-dimensional \Levy\ processes; and Roynette--Vallois--Yor \cite{RVY-d-Brown} and Najnudel--Tansrivorarat \cite{Najnudel-Tansrivorarat}, multidimensional Brownian motion. Conditioning problems have also been studied extensively, although we do not attempt to review that literature here. The two-dimensional {simple random walk} conditioned to avoid the origin, which plays an important role in the present paper, will be introduced in the next section. 

In the discrete-time setting, which is particularly relevant to the present paper, Debs \cite{Debs} studied penalization problems for the one-dimensional {simple random walk} using weights depending on quantities such as the local time at the origin and excursion lengths. Najnudel--Roynette--Yor \cite{Najnudel-Roynette-Yor}, on the other hand, constructed, for more general recurrent Markov chains, a universal $\sigma$-finite measure that is closely related to penalization problems.

The direction-dependent excursion counts considered in the present paper record finer path information near the origin arising from the multidimensional nature of the walk. Indeed, for the one-dimensional {simple random walk}, an excursion that leaves the origin on the $+1$ side necessarily returns to the origin from the $+1$ side, and similarly, an excursion that leaves on the $-1$ side necessarily returns from the $-1$ side. Thus, even if one records the departure and return directions of an excursion separately, no pair of distinct directions can occur. In contrast, for the two-dimensional {simple random walk}, an excursion that leaves the origin in the direction of $v_i$ may return to the origin from the direction of $v_j$ even when $i\neq j$. Consequently, the pair consisting of the departure and return directions, as well as the numbers of crossings of individual edges incident to the origin, contains path information that cannot be captured by the local time at the origin alone.

In this paper, we investigate penalization problems based on such direction-dependent excursion counts and edge-crossing counts. We further show that the penalized measures arising from these finer path functionals can be represented in terms of the same universal $\sigma$-finite measure as that arising from penalization by the local time at the origin.

In (\ref{penalization-limit}), the penalization limit is obtained by letting the deterministic observation time $t$ tend to infinity. Such a choice of observation time is referred to as a \emph{constant clock}. Penalization problems based on random observation times, or \emph{random clocks}, have also been studied. For instance, \cite{Takeda-Yano} considers several random clocks, including an exponential clock based on an exponentially distributed random variable. When a random clock is used, properties of its distribution can often be exploited in the analysis of the penalization limit, and this may make the problem more tractable than its constant-clock counterpart.

In the present paper, we employ a geometric clock based on the geometric distribution. More precisely, for $p\in(0,1)$, set $q:=1-p$, and let $T:=T_p$ be a random variable independent of the random walk $S$, which will be introduced in the next subsection, such that
\begin{align}
\P(T=m)=pq^{m-1}
\qquad \text{for } m\ge 1.
\end{align}
Throughout the paper, expectations involving $T$ are understood to be taken with respect to both the random walk and $T$. We study penalization limits as $p\to0+$, or equivalently, as $q\to1-$. Moreover, in Appendix \ref{AppendixC}, we show that, under several additional assumptions, the penalization limit obtained with the geometric clock coincides with that obtained with the constant clock.

\subsection{Notation}
Hereafter, let $S=(S_n)_{n\ge 0}$ be the two-dimensional {simple random walk}, and let $\P_x$ denote the law of $S$ started from $x\in \Z^2$. We also set
\begin{align}
\F_n:=\sigma(S_0,...,S_n),\qquad \F_\infty:=\sigma\left(\bigcup_{n=0}^\infty \F_n\right).
\end{align}
Let $\tau$ denote the first hitting time of the origin after time $0$:
\begin{align}
\tau:=\inf\Big\{n\ge 1:\ S_n=0\Big\}.
\end{align}
We also let $\rho_0$ denote the first hitting time of the origin including time $0$, and define the successive return times to the origin by
\begin{align}
\rho_0&:=\inf\Big\{n\ge 0:\ S_n=0\Big\},\\
\rho_{k+1}&:=\inf\Big\{n>\rho_k:\ S_n=0\Big\}\qquad \text{for}\ k\ge 0.
\end{align}
Note that, for $x\neq 0$, we have $\tau=\rho_0$ under $\P_x$, whereas $\tau\neq \rho_0$ under $\P_0$. We define the local time $L=(L_n)_{n\ge 0}$ of $S$ at the origin by
\begin{align}
L_n:=\sum_{k=1}^n 1_{\{S_k=0\}}\qquad \text{for}\ n\ge 0.
\end{align}
The local time at the origin records the number of visits to the origin, but does not distinguish the direction in which each excursion leaves the origin or the direction from which it returns to the origin. Moreover, the number of crossings of a given edge incident to the origin cannot, in general, be recovered from the local time at the origin alone.

Motivated by this observation, in this paper we consider two types of penalization problems based on more detailed path information in a neighborhood of the origin. The first is a penalization by the counting process $L^{A\circlearrowleft}$, which counts completed excursions whose pair of departure and return directions belongs to a prescribed set $A$. The second is a penalization by the counting process $N^B$, which counts crossings of unoriented edges belonging to a prescribed set $B$ among the edges incident to the origin. The aim of this paper is to determine explicitly the penalization limits associated with these two types of counting processes and to describe, by means of path decompositions, the path structure of the corresponding penalized processes.

\begin{itemize}
\item Let $v_1,...,v_4$ denote the four lattice points adjacent to the origin:
\begin{align}
v_1:=(1,0),\qquad v_2:=(0,1),\qquad v_3:=(-1,0),\qquad v_4:=(0,-1).
\end{align}
An excursion that leaves the origin in the direction of $v_i$ and returns to the origin from the direction of $v_j$ is called an $(i,j)$-excursion. Fix a nonempty set $A\subseteq \{1,2,3,4\}^2$. An $(i,j)$-excursion satisfying $(i,j)\in A$ is called an $A$-excursion, and we define the counting process $L^{A\circlearrowleft}=(L_n^{A\circlearrowleft})_{n\ge 0}$, which counts the number of completed $A$-excursions, by
\begin{align}
L_n^{A\circlearrowleft}:=\sum_{k= 1}^\infty 1_{\{\rho_k\le n,\ (I_k,J_k)\in A\}}.
\end{align}
Here, $I_k$ and $J_k$ denote the departure and return directions, respectively, of the $k$-th excursion:
\begin{align}
v_{I_k}:=S_{\rho_{k-1}+1},\qquad v_{J_k}:=S_{\rho_k-1}.
\end{align}
We also denote by $p_A$ the probability that a single excursion is an $A$-excursion:
\begin{align}
p_A:=\P_0(L_\tau^{A\circlearrowleft}=1)=\P_0\Big((I_1,J_1)\in A\Big).
\end{align}
The explicit value of $p_A$ is given in Appendix \ref{AppendixB}.

\item For $i=1,2,3,4$, let $e_i:=\{0,v_i\}$ denote the four unoriented edges incident to the origin. Fix a nonempty set $B\subseteq \{e_1,e_2,e_3,e_4\}$, and call an unoriented edge belonging to $B$ a $B$-edge. We define the counting process $N^B=(N_n^B)_{n\ge 0}$, which counts the number of crossings of $B$-edges, by
\begin{align}
N_n^B:=\sum_{k=1}^n 1_{\{\{S_{k-1},S_k\}\in B\}},\qquad N_0^B:=0.
\end{align}
We also denote by $\kappa_B$ the probability that a single excursion crosses a $B$-edge at least once:
\begin{align}
\kappa_B:=\P_0(N_\tau^B\ge 1).
\end{align}
For the explicit value of $\kappa_B$, see Table \ref{table} {in Appendix \ref{AppendixB}}.
\end{itemize}

We define the Green's generating function $G_x$ for the two-dimensional {simple random walk} by
\begin{align}
\label{def-Green}
G_{x}(\xi):=\sum_{n=0}^\infty \P_x(S_n=0)\xi^n\qquad \text{for}\ x\in \Z^2,\ 0\le \xi< 1.
\end{align}
It is well known that the following limit exists and is finite:
\begin{align}
\label{a-definition}
a(x):=\lim_{\xi\to 1-}\Big(G_0(\xi)-G_x(\xi)\Big)\qquad \text{for}\ x\in \Z^2
\end{align}
(see, e.g., Section 4.4.1 of \cite{Lawler-Limic}). This function is called the potential kernel. The potential kernel $a$ provides the harmonic function used to construct the two-dimensional {simple random walk} conditioned to avoid the origin. More precisely, for $x\neq 0$, a probability measure $\P_x^\updownarrow$ on $\F_\infty$ can be defined by
\begin{align}
\frac{d\P_x^{\updownarrow}}{d\P_x}\Big|_{\F_n}=\frac{a(S_n)}{a(x)}1_{\{n<\tau\}}.
\end{align}
This is the h-transform by $a$ of the two-dimensional {simple random walk} killed at the origin, and it satisfies
\begin{align}
\P_x^\updownarrow(\tau=\infty)=1.
\end{align}
Thus, the process under $\P_x^\updownarrow$ can be interpreted as the two-dimensional {simple random walk} conditioned to avoid the origin. For this process, see, for example, Comets--Popov--Vachkovskaia \cite{Comets-Popov-Vachkovskaia} and Section 4 of \cite{Popov}. In the path decompositions of the penalized processes obtained in this paper, the path after its final departure from the origin is described by this law $\P_x^\updownarrow$.

\subsection{Main Results}
We now state the main results. Explicit expressions for the functions appearing in the limits, as well as the proofs of the theorems, will be given in later sections.

Unless otherwise stated, throughout the remainder of the paper, the function $f:\Z_{\ge 0}\to [0,\infty)$ is assumed to satisfy
\begin{align}
\label{assumption-f}
0<\sum_{k=0}^\infty f(k)<\infty.
\end{align}

\subsubsection{Penalization by Direction-Dependent Excursion Counts}
We introduce the process $(V_n)_{n\ge 0}$ recording the departure direction of the excursion currently in progress:
\begin{align}
V_n&:=\begin{cases}
S_{\rho_{k}+1}&\text{if}\ \rho_k<n<\rho_{k+1}\ \text{for}\ k\ge 0,\\[10pt]
0&\text{if}\ S_n=0\ \text{or}\ 0\le n<\rho_0.
\end{cases}
\end{align}
We also set
\begin{align}
E:=\Big(\Z^2\times \{0\}\Big)\cup\Big((\Z^2\setminus \{0\})\times \{v_1,...,v_4\}\Big).
\end{align}
For $(x,v)\in E$, let $\P_{x,v}$ denote the law of the process $(S,L^{A\circlearrowleft},V)$ started from $(x,0,v)$. When $v=v_i$, time $0$ is interpreted as lying in the middle of an excursion that departed from the origin in the direction of $v_i$. In this case, we set
\begin{align}
V_n=v_i\ \text{for}\ 0\le n<\rho_0,\qquad V_{\rho_0}=0,
\end{align}
and, if the return direction of this excursion is $v_j$, then $L^{A\circlearrowleft}$ is increased by $1$ at time $\rho_0$ whenever $(i,j)\in A$. After time $\rho_0$, the process evolves according to the rules defined above. When $v=0$, we use the usual definition and, in particular, identify $\P_{x,0}$ with $\P_x$.

We then obtain the following penalization limit.

\begin{thm}
\label{thm1.1}
For $(x,v)\in E$ and $l\in \Z_{\ge 0}$, 
\begin{align}
\label{PxvfLTA-lim}
\lim_{q\to 1-}p_AG_0(q)\P_{x,v}\Big[f(L_T^{A\circlearrowleft}+l)\Big]&= h_f^A(x,l,v).
\end{align}
Moreover,
\begin{align}
\Big(h_f^A(S_n,L_n^{A\circlearrowleft},V_n)\Big)_{n\ge 0}
\end{align}
is a nonnegative $\P_{x,v}$-martingale, and
\begin{align}
\label{penal-limit}
\lim_{q\to 1-}\frac{\P_x[1_{\Lambda_n}\cdot f(L_{T}^{A\circlearrowleft})]}{\P_x[f(L_{T}^{A\circlearrowleft})]}=\P_x\left[1_{\Lambda_n}\cdot \frac{h_f^A(S_n,L_n^{A\circlearrowleft},V_n)}{h_f^{A}(x,0,0)}\right]
\end{align}
for $n\ge 0,\ x\in \Z^2,$ and $\Lambda_n\in \F_n$. In particular,
\begin{align}
h_f^A(x,0,0)>0
\end{align}
for every $x\in \Z^2$.
\end{thm}

The function $h_f^A$ is defined in (\ref{Def-h_f^A}), and the proof of this theorem will be given in Subsection \ref{S2.1}. More precisely, the assertions of this theorem are proved separately in Lemmas \ref{lem2.3} and \ref{h-martingale-lem} and Proposition \ref{thm2.5}.

For $x\in \Z^2$, let $\Q_x^{f,A\circlearrowleft}$ denote the penalized probability measure on $\F_\infty$ arising in the limit (\ref{penal-limit}); that is,
\begin{align}
\frac{d\Q_x^{f,A\circlearrowleft}}{d\P_x}\Big|_{\F_n}=\frac{h_f^A(S_n,L_n^{A\circlearrowleft},V_n)}{h_f^A(x,0,0)}
\qquad \text{for}\ n\ge 0,\ x\in \Z^2.
\end{align}
The existence of such a measure follows from, for example, Theorem A.1 of \cite{Yano-Penal}. By Lemma \ref{h-martingale-lem}, $h_f^A$ is harmonic for the process $(S,L^{A\circlearrowleft},V)$. Hence, under the probability measure $\Q_x^{f,A\circlearrowleft}$, the process $(S,L^{A\circlearrowleft},V)$ is the h-transform of $(S,L^{A\circlearrowleft},V)$ with respect to $h_f^A$. Since $(S,L^{A\circlearrowleft},V)$ is a strong Markov process, the general theory of h-transforms implies that $(S,L^{A\circlearrowleft},V)$ under $\Q_x^{f,A\circlearrowleft}$ is also a strong Markov process (see, e.g., Theorem 11.9 of \cite{C-W}).

For this penalized process, we obtain an explicit path decomposition in terms of the total number of $A$-excursions. Here we give only an overview; a precise formulation and proof will be given later in Subsection \ref{S2.2}.

Under $\Q_x^{f,A\circlearrowleft}$, the total number of $A$-excursions $L_\infty^{A\circlearrowleft}$ is almost surely finite. Let $k\ge 1$ and suppose that the event $\{L_\infty^{A\circlearrowleft}=k\}$ has positive probability. Conditional on this event, the path up to the completion time $\sigma_k$ of the $k$-th $A$-excursion has the same law as the original simple random walk stopped at $\sigma_k$. The path after $\sigma_k$ has law $\P_{0,0}^{A\updownarrow}$ and is independent of the path before $\sigma_k$. On the other hand, if $\{L_\infty^{A\circlearrowleft}=0\}$ has positive probability, then, conditional on this event, the entire path has law $\P_{x,0}^{A\updownarrow}$ (Proposition \ref{thm2.9}). Here, $\P_{x,0}^{A\updownarrow}$ denotes the law of the process started from $x$ and conditioned never to complete an $A$-excursion; its precise construction is given in (\ref{def-P_{xv}^Aupdown}).

Furthermore, under $\P_{0,0}^{A\updownarrow}$, the path can be decomposed into finitely many excursions followed by a path that never returns to the origin. If $A\neq\{1,2,3,4\}^2$, the number of these finite excursions has a geometric distribution on $\Z_{\ge 0}$ with parameter $p_A$. Conditional on this number, the excursions are independent, and each has the law of an excursion of the original simple random walk conditioned on $(I_1,J_1)\notin A$. After these excursions, the process moves to a point $v_i$ chosen uniformly from the four points adjacent to the origin and, from that point onward, follows the law $\P_{v_i}^{\updownarrow}$ (Proposition \ref{thm2.10}). If $A=\{1,2,3,4\}^2$, no finite excursion is made: on its first step, the process moves to a point $v_i$ chosen uniformly from the four points adjacent to the origin and, from that point onward, follows the law $\P_{v_i}^\updownarrow$ (Proposition \ref{thm2.11}).

Finally, let $x\neq 0$. Under $\P_{x,0}^{A\updownarrow}$, we distinguish according to whether or not the path hits the origin. Conditional on hitting the origin, the path up to the hitting time $\tau$ has the same law as the original simple random walk, while the path after $\tau$ follows $\P_{0,0}^{A\updownarrow}$; under this conditioning, the pre-$\tau$ and post-$\tau$ paths are independent. On the other hand, conditional on never hitting the origin, the entire path has law $\P_x^\updownarrow$ (Proposition \ref{thm2.12}).

The following figure illustrates the path decomposition in the case $A\neq \{1,2,3,4\}^2$, under the conditions that $L_\infty^{A\circlearrowleft}=k\ge 1$ under $\Q_x^{f,A\circlearrowleft}$ and $L_\infty=n\ge 1$ under $\P_{0,0}^{A\updownarrow}$. 

\begin{figure}[htbp]
\centering
\begin{tikzpicture}[
    >=Latex,
    font=\small,
    x=1cm,
    y=1cm,
    rw/.style={
      thick,
      decorate,
      decoration={snake,amplitude=.7mm,segment length=4mm}
    },
    tail/.style={
      thick,
      dashed,
      decorate,
      decoration={snake,amplitude=.7mm,segment length=4mm}
    },
    excursion/.style={thick},
    guide/.style={gray!60}
]

\node[anchor=west] at (-0.3,5.25)
{\textbf{(a)}\quad
Conditional on $L_{\infty}^{A\circlearrowleft}=k\ge1$};

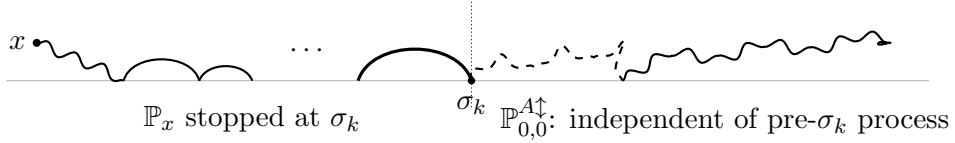
\draw[guide] (0,3.5) -- (12.2,3.5);

\fill (0.4,4.0) circle (1.5pt);
\node[left] at (0.4,4.0) {$x$};

\draw[rw]
(0.4,4.0) --
(1.1,3.72) --
(1.55,3.5);

\draw[excursion]
(1.55,3.5)
to[out=75,in=105] (2.55,3.5);

\draw[excursion]
(2.55,3.5)
to[out=70,in=110] (3.25,3.5);

\node at (4.0,3.9) {$\cdots$};

\draw[excursion, line width=1.2pt]
(4.65,3.5)
to[out=72,in=108]
node[above=5pt] {}
(6.15,3.5);

\fill (6.15,3.5) circle (1.4pt);
\node[below=2pt] at (6.15,3.5) {$\sigma_k$};

\draw[densely dotted]
(6.15,3.15) -- (6.15,4.55);

\draw[tail]
(6.15,3.5)
to[out=65,in=115] (6.9,3.5)
to[out=60,in=120] (7.55,3.5)
to[out=60,in=120] (8.15,3.5);

\draw[rw]
(8.15,3.5) --
(8.7,3.85) --
(9.25,3.72) --
(9.8,3.97) --
(10.4,3.82) --
(11.0,4.05) --
(11.7,4.0);

\node[align=center] at (3.25,3)
{$\mathbb P_x$ stopped at $\sigma_k$};

\node[align=center] at (9.5,3)
{$\mathbb P^{A\updownarrow}_{0,0}$: independent of pre-$\sigma_k$ process};

\node[anchor=west] at (-0.3,1.85)
{\textbf{(b)}\quad
Under $\mathbb P^{A\updownarrow}_{0,0}$,
conditional on $L_\infty=n\ge1$};

\draw[guide] (0,0.2) -- (12.2,0.2);

\fill (0.6,0.2) circle (1.3pt);
\node[below=2pt] at (0.6,0.2) {$0$};

\draw[dashed,thick]
(0.6,0.2)
to[out=65,in=115]
node[above=3pt] {$\mathcal E_1$}
(2.0,0.2);

\node at (3.2,0.65) {$\cdots$};

\draw[dashed,thick]
(4.1,0.2)
to[out=65,in=115]
node[above=3pt] {$\mathcal E_n$}
(5.6,0.2);

\fill (5.6,0.2) circle (1.3pt);
\node[below=2pt] at (5.6,0.2) {$0$};

\draw[thick] (5.6,0.2) -- (6.45,0.65);
\fill (6.45,0.65) circle (1.3pt);
\node[above=2pt] at (6.45,0.65) {$v_i$};

\draw[rw]
(6.45,0.65) --
(7.1,0.87) --
(7.7,0.75) --
(8.25,1.0) --
(8.9,0.88) --
(9.55,1.1) --
(10.2,1.0) --
(10.9,1.18) --
(11.6,1.15);

\node[align=center] at (3.05,-0.75)
{$n$ independent $A^c$-excursions\\
each with law
$\mathbb P_0\!\left(\,\cdot\mid(I_1,J_1)\notin A\right)$};

\node[align=left] at (8.8,-0.65)
{$\mathbb P^{\updownarrow}_{v_i}$: independent of $\mathcal E_1,\ldots,\mathcal E_n$};

\end{tikzpicture}
\caption{Path decomposition under
$\mathbb Q_x^{f,A\circlearrowleft}$.}
\label{fig:A-path-decomposition}
\end{figure}

\subsubsection{Penalization by Edge-Crossing Counts}
For the edge-crossing count $N^B$, we obtain the following penalization limit.

\begin{thm}
\label{thm1.2}
For $x\in \Z^2$ and $l\in \Z_{\ge 0}$, 
\begin{align}
\label{thm1.2-eq1}
\lim_{q\to 1-}G_0(q)\P_x\Big[f(N_T^B+l)\Big] = H_f^B(x,l).
\end{align}
Moreover,
\begin{align}
\Big(H_f^B(S_n,N_n^B)\Big)_{n\ge 0}
\end{align}
is a nonnegative $\P_x$-martingale, and, for every $x\in \Z^2$ such that $H_f^B(x,0)>0$,
\begin{align}
\label{penal-limit-2}
\lim_{q\to 1-}\frac{\P_x[1_{\Lambda_n}\cdot f(N_{T}^{B})]}{\P_x[f(N_{T}^{B})]}=\P_x\left[1_{\Lambda_n}\cdot \frac{H_f^B(S_n,N_n^B)}{H_f^B(x,0)}\right]
\end{align}
for $n\ge 0$ and $\Lambda_n\in \F_n$.
\end{thm}

The function $H_f^B$ is defined in (\ref{def-HfB}), and the proof of this theorem will be given in Subsection \ref{S3.1}. More precisely, the assertions of this theorem are proved separately in Propositions \ref{3.6} and \ref{thm3.7}. Moreover, Proposition \ref{lem3.8} gives an equivalent condition for $H_f^B(x,0)>0$.

For $x\in \Z^2$ satisfying $H_f^B(x,0)>0$, let $\Q_x^{f,B\leadsto}$ denote the penalized probability measure on $\F_\infty$ arising in the limit (\ref{penal-limit-2}); that is,
\begin{align}
\frac{d\Q_x^{f,B\leadsto}}{d\P_x}\Big|_{\F_n}:=\frac{H_f^B(S_n,N_n^B)}{H_f^B(x,0)}\qquad \text{for}\ n\ge 0.
\end{align}
As above, under $\Q_x^{f,B\leadsto}$, the process $(S,N^B)$ is the h-transform of $(S,N^B)$ with respect to $H_f^B$ and is a strong Markov process. In what follows, we consider the penalized measure associated with the edge-crossing count only for starting points $x\in \Z^2$ satisfying $H_f^B(x,0)>0$.

In this case as well, we obtain an explicit path decomposition in terms of the total number of crossings of $B$-edges. Here we give only an overview; a precise formulation and proof will be given in Subsection \ref{S3.2}.

Suppose that $B\neq\{e_1,e_2,e_3,e_4\}$. Under $\Q_x^{f,B\leadsto}$, $N_\infty^B$ is almost surely finite. Let $\varsigma_k$ denote the time of the $k$-th crossing of a $B$-edge. For $k\ge 1$ and $y\in \Z^2$, suppose that the event $\{N_\infty^B=k,\ S_{\varsigma_k}=y\}$ has positive probability. Conditional on this event, the path up to $\varsigma_k$ has the same law as the original simple random walk stopped at $\varsigma_k$, conditional on $S_{\varsigma_k}=y$. The path after $\varsigma_k$ has law $\mathbb P_y^{B\Updownarrow}$ and is independent of the path before $\varsigma_k$. On the other hand, if $\{N_\infty^B=0\}$ has positive probability, then, conditional on this event, the entire path has law $\P_x^{B\Updownarrow}$ (Proposition \ref{thm3.10}). Here, $\P_x^{B\Updownarrow}$ denotes the law of the process started from $x$ and conditioned never to cross a $B$-edge; its precise construction is given in (\ref{def-PxBupd}).

The path under $\P_y^{B\Updownarrow}$ can be decomposed further. First, suppose that $y\neq 0$. We distinguish according to whether or not the path hits the origin. Conditional on hitting the origin, the path up to the hitting time $\tau$ follows the law of the simple random walk conditioned on $N_\tau^B=0$. The path after $\tau$ is independent of the path before $\tau$ and has law $\P_0^{B\Updownarrow}$. On the other hand, conditional on never hitting the origin, the entire path has law $\P_y^{\updownarrow}$.

Next, under $\P_0^{B\Updownarrow}$, the process makes finitely many excursions that do not cross any $B$-edge and then leaves the origin for the last time. The number of such excursions has a geometric distribution on $\Z_{\ge 0}$ with parameter $\kappa_B$. Conditional on this number, the excursions are independent, and each has the law of an ordinary excursion conditioned on $N_\tau^B=0$. The process then leaves the origin along an edge chosen uniformly from $\{e_1,e_2,e_3,e_4\}\setminus B$. If the endpoint of this edge is $v_i$, then, from that point onward, the path follows the law $\P_{v_i}^\updownarrow$ (Proposition \ref{thm3.11}).

On the other hand, suppose that $B=\{e_1,e_2,e_3,e_4\}$. For a path with a positive total number of crossings, the position immediately after the final crossing of a $B$-edge is one of the neighboring points $v_i$ distinct from the origin. From that point onward, the path follows the law $\P_{v_i}^{\updownarrow}$. Moreover, if the process starts from $x\neq 0$ and the total number of crossings is $0$, then the entire path has law $\P_x^\updownarrow$ (Proposition \ref{3.13}).

The following figure illustrates the path decomposition in the case $B\neq \{e_1,e_2,e_3,e_4\}$, under the conditions that $N_\infty^B=k\ge 1$ under $\Q_x^{f,B\leadsto}$ and $L_\infty=n\ge 1$ under $\P_{0}^{B\Updownarrow}$.

\begin{figure}[htbp]
\centering
\begin{tikzpicture}[
    >=Latex,
    font=\small,
    x=1cm,
    y=1cm,
    rw/.style={
      thick,
      decorate,
      decoration={snake,amplitude=.7mm,segment length=4mm}
    },
    tail/.style={
      thick,
      dashed,
      decorate,
      decoration={snake,amplitude=.7mm,segment length=4mm}
    },
    guide/.style={gray!60}
]

\node[anchor=west] at (-0.3,6.0)
{\textbf{(a)}\quad
Conditional on $N_\infty^B=k\ge1$ and $S_{\varsigma_k}=y$};

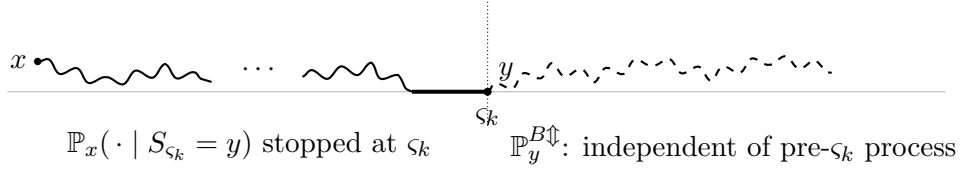
\draw[guide] (0,4.25) -- (12.4,4.25);

\fill (0.4,4.65) circle (1.4pt);
\node[left] at (0.4,4.65) {$x$};

\draw[rw]
(0.4,4.65) --
(1.2,4.4) --
(2.0,4.58) --
(2.7,4.38);

\node at (3.35,4.55) {$\cdots$};

\draw[rw]
(3.9,4.45) --
(4.7,4.55) --
(5.35,4.25);

\draw[line width=1.5pt]
(5.35,4.25) -- (6.35,4.25);

\fill (6.35,4.25) circle (1.5pt);
\node[above right] at (6.35,4.25) {$y$};
\node[below=2pt] at (6.35,4.25) {$\varsigma_k$};

\draw[densely dotted]
(6.35,3.8) -- (6.35,5.45);

\draw[tail]
(6.35,4.25) --
(7.1,4.52) --
(7.8,4.38) --
(8.5,4.58) --
(9.2,4.4) --
(10.0,4.62) --
(10.9,4.5);

\node[align=center] at (3.2,3.55)
{$\mathbb P_x(\,\cdot\mid S_{\varsigma_k}=y)$ stopped at $\varsigma_k$};

\node[align=center] at (9.6,3.55)
{$\mathbb P_y^{B\Updownarrow}$: independent of pre-$\varsigma_k$ process};

\node[anchor=west] at (-0.3,2.65)
{\textbf{(b)}\quad
Tail from $y\neq0$: law $\mathbb P_y^{B\Updownarrow}$};

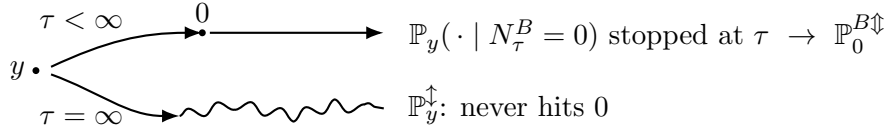
\fill (0.8,1.25) circle (1.3pt);
\node[left] at (0.8,1.25) {$y$};

\draw[->,thick]
(1.0,1.3)
to[out=25,in=180]
(3.0,1.75);

\node[above] at (1.4,1.65) {$\tau<\infty$};

\fill (3.0,1.75) circle (1.3pt);
\node[above] at (3.0,1.75) {$0$};

\draw[->,thick] (3.1,1.75) -- (5.4,1.75);

\node[left] at (12.18,1.75)
{$\P_y(\,\cdot\mid N_\tau^B=0)$ stopped at $\tau\ \to\ \P_0^{B\Updownarrow}$ };

\draw[->,thick]
(1.0,1.2)
to[out=-20,in=180]
(2.7,0.65);

\node[below] at (1.4,0.85) {$\tau=\infty$};

\draw[rw]
(2.7,0.65) --
(3.4,0.8) --
(4.05,0.67) --
(4.75,0.86) --
(5.4,0.75);

\node at (7.05,0.8)
{$\mathbb P_y^{\updownarrow}$: never hits $0$};

\node[anchor=west] at (-0.3,-0.35)
{\textbf{(c)}\quad
Under $\mathbb P_0^{B\Updownarrow}$,
conditional on $L_\infty=n\ge1$};

\draw[guide] (0,-2.0) -- (12.3,-2.0);

\fill (0.6,-2.0) circle (1.3pt);
\node[below=2pt] at (0.6,-2.0) {$0$};

\draw[dashed,thick]
(0.6,-2.0)
to[out=65,in=115]
node[above=3pt] {$\mathcal E_1$}
(2.0,-2.0);

\node at (3.1,-1.55) {$\cdots$};

\draw[dashed,thick]
(4.0,-2.0)
to[out=65,in=115]
node[above=3pt] {$\mathcal E_n$}
(5.5,-2.0);

\fill (5.5,-2.0) circle (1.3pt);
\node[below=2pt] at (5.5,-2.0) {$0$};

\draw[line width=1.2pt]
(5.5,-2.0) -- (6.4,-1.5);

\fill (6.4,-1.5) circle (1.3pt);
\node[above=2pt] at (5.9,-1.52) {$e_i\notin B$};

\draw[rw]
(6.4,-1.5) --
(7.1,-1.27) --
(7.8,-1.4) --
(8.45,-1.18) --
(9.1,-1.3) --
(9.8,-1.08) --
(10.5,-1.18) --
(11.2,-1.02);

\node[align=center] at (3.0,-2.95)
{$n$ independent excursions\\
$\mathbb P_0(\,\cdot\mid N_\tau^B=0)$ stopped at $\tau$};

\node[align=left] at (8.75,-2.8)
{$\mathbb P_{v_i}^{\updownarrow}$: independent of $\mathcal E_1,\ldots,\mathcal E_n$
};

\end{tikzpicture}
\caption{Path decomposition under
$\mathbb Q_x^{f,B\leadsto}$.}
\label{fig:B-path-decomposition}
\end{figure}

{
We next state another main result.

\begin{thm}[Parts of Propositions \ref{thm2.10} and \ref{thm3.11}]
Assume that $A\neq\{1,2,3,4\}^2$ and $|B| < 4$. For $n\ge 0$, we have 
\begin{align}
    \P_{0,0}^{A\updownarrow}(L_\infty=n)&=p_A(1-p_A)^{n}, \\
   \P_0^{B\Updownarrow}(L_\infty=n)&=\kappa_B(1-\kappa_B)^n.
\end{align}
In particular, 
\begin{equation}
\P_{0,0}^{A\updownarrow}(L_\infty < \infty) = \P_0^{B\Updownarrow}(L_\infty<\infty) = 1.
\end{equation}
\end{thm}
This theorem seems to be non-trivial. 
Under the assumptions on $A$ and $B$, the simple random walk conditioned to avoid $A$ (or $B$) may return to the origin infinitely often while avoiding $A$ (or $B$). 
However, this theorem states that the conditioned walks return to the origin only a finite number of times. 
This phenomenon cannot be demonstrated by observing $\P_x^\updownarrow$ alone. 
It can only be demonstrated by observing $\P_{0,0}^{A\updownarrow}$ or $\P_0^{B\Updownarrow}$. }

\subsubsection{Universal $\sigma$-Finite Measure}
In Theorems \ref{thm1.1} and \ref{thm1.2}, we considered penalization problems with different path functionals as weights, and at first sight obtained different penalized measures. One of the main conclusions of this paper, however, is that these two types of penalized measures can in fact be represented in a unified way in terms of the same $\sigma$-finite measure.

\begin{thm}
\label{thm1.3}
There exists a $\sigma$-finite measure $\widetilde{\mathscr{Q}}_x$ on $\F_\infty$ such that
\begin{align}
\Q_x^{f,A\circlearrowleft}=\frac{p_Af(L_\infty^{A\circlearrowleft})}{h_f^A(x,0,0)}\cdot \widetilde{\mathscr{Q}}_x\qquad \text{on}\ \F_\infty.
\end{align}
Moreover, for every $x$ satisfying $H_f^B(x,0)>0$,
\begin{align}
\Q_x^{f,B\leadsto}=\frac{f(N_\infty^B)}{H_f^B(x,0)}\cdot \widetilde{\mathscr{Q}}_x\qquad \text{on}\ \F_\infty.
\end{align}
\end{thm}

The proof of this theorem is given in Section \ref{S4}, and the measure $\widetilde{\mathscr{Q}}_x$ is defined in (\ref{NRV-universal}).

The measure $\widetilde{\mathscr{Q}}_x$ is the universal $\sigma$-finite measure arising from the penalization problem associated with the local time at the origin. Consequently, the penalizations considered in this paper, namely those based on direction-dependent excursion counts, edge-crossing counts, and the local time at the origin, can all be represented in terms of the same measure $\widetilde{\mathscr{Q}}_x$.

More generally, when penalized measures corresponding to different weights can be represented by weighting and normalizing the same universal $\sigma$-finite measure, the corresponding weights are said to belong to the same \emph{universality class}. Therefore, the three types of penalization considered in this paper belong to the same universality class.

\subsubsection*{Organization}
The remainder of this paper is organized as follows. Section \ref{S2} is devoted to penalization by direction-dependent excursion counts. We establish the geometric-clock penalization limit and derive path decomposition formulas for the corresponding penalized process. Section \ref{S3} develops the analogous results for penalization by edge-crossing counts. In Section \ref{S4}, we show that both families of penalized measures can be represented in terms of the same universal $\sigma$-finite measure associated with local-time penalization at the origin. Appendix \ref{AppendixA} presents the corresponding results for penalization by local time at the origin and compares them with special cases of excursion-count and edge-crossing-count penalizations. Appendix \ref{AppendixC} shows that the constant-clock and geometric-clock penalization limits coincide when the weight function is nonincreasing or has finite support. Finally, Appendix \ref{AppendixB} provides explicit values of the constants appearing in the paper.


\section{Penalization by Direction-Dependent Excursion Counts}
\label{S2}
In this section, we study the penalization problem with weights of the form $\Gamma=f(L^{A\circlearrowleft})$ and derive a path decomposition for the penalized process.

\subsection{Penalization Limit}
\label{S2.1}
In this subsection, we prove Theorem \ref{thm1.1}.

Let $F_x$ denote the generating function of $\tau$ under $\mathbb P_x$, and let $F^A$ denote the generating function of the excursion length restricted to the event that the excursion is an $A$-excursion. More precisely, we define
\begin{align}
F_x(\xi)&:=\P_x[\xi^\tau]=\sum_{n=1}^\infty \mathbb{P}_x(\tau=n)\xi^n\qquad \text{for}\ 0\le \xi\le 1\ \text{and}\ x\in \Z^2,\\
  F^{A}(\xi)&:=\P_0\Big[\xi^{\tau},\ (I_1,J_1)\in A\Big]=\sum_{n=1}^\infty \P_0\Big(\tau=n,\ (I_1,J_1)\in A\Big)\xi^n\qquad \text{for}\ 0\le \xi\le 1.
 \end{align}
Note that $F^A(1)=p_A>0$, since $A$ is nonempty. Set $\sigma_0:=\rho_0$. {In addition, note that $|F_0(\xi)-F^{A}(\xi)| < 1$ for any $\xi \in [0, 1]$. } For $k\ge 1$, let $\sigma_k$ denote the completion time of the $k$-th $A$-excursion:
\begin{align}
  \sigma_k:=\inf\Big\{\rho_m>\sigma_{k-1}:\ (I_m,J_m)\in A\Big\}.
\end{align}

We first determine the distribution of $L_T^{A\circlearrowleft}$.

\begin{lem}
\label{lem2.1}
The generating function of $\sigma_1$ under $\P_0$ is given by
\begin{align}
\label{lem2.1-eq2}
 \P_0[\xi^{\sigma_1}]=\frac{F^{A}(\xi)}{1-F_0(\xi)+F^{A}(\xi)}=:\widetilde{F}^A(\xi)\qquad \text{for}\ 0\le \xi\le 1.
\end{align}
Moreover, we have
 \begin{align}
 \label{lem2.1-eq3}
   \P_0(L_T^{A\circlearrowleft}=0)&=1-\frac{\widetilde{F}^A(q)}{q},\\
   \label{lem2.1-eq4}
\P_0(L_T^{A\circlearrowleft}=k)&=\frac{1-\widetilde{F}^A(q)}{q}\cdot \widetilde{F}^A(q)^{k} \qquad \text{for}\ k\ge 1.
 \end{align}
If $x\neq 0$, then we have
\begin{align}
\label{lem2.1-eq5}
\P_x(L_T^{A\circlearrowleft}=0)&=1-\frac{F_x(q)\widetilde{F}^A(q)}{q},\\
\label{lem2.1-eq6}
\P_x(L_T^{A\circlearrowleft}=k)&=\frac{1-\widetilde{F}^A(q)}{q}\cdot F_x(q)\widetilde{F}^A(q)^{k} \qquad \text{for}\ k\ge 1.
\end{align}
\end{lem}
\begin{proof}
By the strong Markov property, we have
 \begin{align*}
\P_0[\xi^{\sigma_1}]&=\sum_{m=1}^\infty \P_0\Big[\xi^{\sigma_1},\ \sigma_1=\rho_m \Big]=\sum_{m=1}^\infty \left(\P_0\Big[\xi^{\rho_1},\ \sigma_1\neq \rho_1\Big]\right)^{m-1}\P_0\Big[\xi^{\rho_1},\ \sigma_1=\rho_1\Big]\\
&=\sum_{m=1}^\infty \Big(F_0(\xi)-F^{A}(\xi)\Big)^{m-1}F^{A}(\xi)=\frac{F^{A}(\xi)}{1-F_0(\xi)+F^{A}(\xi)}
\stepcounter{equation}\tag{\theequation}
\end{align*}
for $0\le \xi\le 1.$ This proves (\ref{lem2.1-eq2}).

By the strong Markov property, the increments $(\sigma_k-\sigma_{k-1})_{k\ge1}$ are independent and identically distributed under $\P_0$. Thus, for $k\ge 1$, we have
{\begin{align*}
\label{lem2.1-eq1}
\P_0(L_T^{A\circlearrowleft}\ge k)&=\P_0(\sigma_k\le T)=\sum_{m=1}^\infty \P_0(m \le T,\ \sigma_k = m) = \sum_{m=1}^{\infty} \P_0(m \le T) \P_0(\sigma_k = m) \\
&= \sum_{m=1}^{\infty} q^{m-1}\P_0(\sigma_k = m) = \frac{1}{q}\P_0[q^{\sigma_k}]
=\frac{1}{q}\Big(\P_0[q^{\sigma_1}]\Big)^k=\frac{1}{q}\widetilde{F}^A(q)^k.
\stepcounter{equation}\tag{\theequation}
\end{align*}}
Therefore, we obtain
 \begin{align}
   \P_0(L_T^{A\circlearrowleft}=0)&=1-\P_0(L_T^{A\circlearrowleft}\ge 1)=1-\frac{1}{q}\widetilde{F}^A(q),\\
\P_0(L_T^{A\circlearrowleft}=k)&=\P_0(L_T^{A\circlearrowleft}\ge k)-\P_0(L_T^{A\circlearrowleft}\ge k+1)=\frac{1-\widetilde{F}^A(q)}{q}\cdot\widetilde{F}^A(q)^k\qquad \text{for}\ k\ge 1.
 \end{align}
This proves (\ref{lem2.1-eq3}) and (\ref{lem2.1-eq4}).

For $x\neq 0$, the strong Markov property at $\tau$ gives
\begin{align}
\P_x[q^{\sigma_k}]&=\P_x[q^{\tau}]\Big(\P_0[q^{\sigma_1}]\Big)^k=F_x(q)\widetilde{F}^A(q)^k\qquad \text{for}\ k\ge 1.
 \end{align}
Equations (\ref{lem2.1-eq5}) and (\ref{lem2.1-eq6}) now follow by the same argument as above. This completes the proof.
\end{proof}

The generating functions $G_0$ and $F_0$ satisfy
\begin{align}
\label{G_0F_0-identity}
  G_0(\xi)=\frac{1}{1-F_0(\xi)}\qquad \text{for}\ 0\le \xi<1
\end{align}
(see, e.g., Proposition 4.2.2 of \cite{Lawler-Limic}). Moreover, for $x\neq 0$, we have
\begin{align*}
\label{GxFx-identity}
G_x(\xi)&=\sum_{n=0}^\infty \P_x(S_n=0)\xi^n=\sum_{n=0}^\infty\sum_{k=1}^n \P_x(\tau=k)\P_0(S_{n-k}=0)\xi^n\\
&=\sum_{k=1}^\infty\P_x(\tau=k)\xi^k\cdot \sum_{n=k}^\infty \P_0(S_{n-k}=0)\xi^{n-k}=F_x(\xi)\cdot G_0(\xi).
\stepcounter{equation}\tag{\theequation}
\end{align*}

Using these relations, we establish the following limit, which is the special case $l=0$ and $v=0$ of (\ref{PxvfLTA-lim}).

\begin{lem}
For $x\in \Z^2$, we have
 \begin{align}
 \label{lem2.2-eq3}
   \lim_{q\to 1-}p_AG_0(q)\P_x\Big[f(L_T^{A\circlearrowleft})\Big]&=\Big(p_Aa(x)+1\Big)f(0)+\sum_{k=1}^\infty f(k).
 \end{align}
\end{lem}
\begin{proof}
By (\ref{G_0F_0-identity}), we have
\begin{align}
\label{tildeFA-identity}
  \widetilde{F}^A(q)=\frac{F^{A}(q)}{1-F_0(q)+F^{A}(q)}=\frac{G_0(q)F^{A}(q)}{1+G_0(q)F^{A}(q)}.
\end{align}
Thus, by Lemma \ref{lem2.1}, we have
\begin{align*}
\label{lem2.2-eq1}
G_0(q)\P_0\Big[f(L_T^{A\circlearrowleft})\Big]&=G_0(q)\P_0(L_T^{A\circlearrowleft}=0)f(0)+G_0(q)\sum_{k=1}^\infty\P_0(L_T^{A\circlearrowleft}=k)f(k)\\ 
&=G_0(q)\left(1-\frac{\widetilde{F}^A(q)}{q}\right)f(0)+G_0(q)\cdot\frac{1-\widetilde{F}^A(q)}{q}\sum_{k=1}^\infty f(k)\widetilde{F}^A(q)^{k}\\
   &=\frac{q+(q-1)G_0(q)F^{A}(q)}{q(\frac{1}{G_0(q)}+F^{A}(q))}\cdot f(0)+\frac{1}{q(\frac{1}{G_0(q)}+F^{A}(q))}\sum_{k=1}^\infty f(k)\widetilde{F}^A(q)^{k}.
     \stepcounter{equation}\tag{\theequation}
\end{align*}
By (\ref{GxFx-identity}), we have
\begin{align*}
\label{lem2.2-eq2}
G_0(q)\P_x\Big[f(L_T^{A\circlearrowleft})\Big]&=\frac{q+F^A(q)\{q(G_0(q)-G_x(q))+(q-1)G_x(q)\}}{q(\frac{1}{G_0(q)}+F^A(q))}\cdot f(0)\\
&\qquad +\frac{F_x(q)}{q(\frac{1}{G_0(q)}+F^A(q))}\sum_{k=1}^\infty f(k)\widetilde{F}^A(q)^k
\stepcounter{equation}\tag{\theequation}
\end{align*}
for $x\neq 0$.

The asymptotic relation
 \begin{align}
 \label{G-log(1-q)}
   G_0(q)\sim \frac{1}{\pi}\log \frac{1}{1-q}\qquad \text{as}\ q\to 1-
 \end{align}
(see, e.g., p. 80 of \cite{Lawler-Limic}) implies that
\begin{align}
   \label{G_0(q)sim}
  \lim_{q\to 1-}G_0(q)=\infty,\qquad   \lim_{q\to 1-}(1-q)G_0(q)=0.
\end{align}
Note that
\begin{align}
   \lim_{q\to 1-}F^A(q)=p_A,\qquad \lim_{q\to 1-}\widetilde{F}^A(q)=1, {\qquad \lim_{q\to 1-} F_x(q) = 1}.
\end{align}
Since $0\le\widetilde{F}^A(q)\le 1$ and $f$ is summable, the dominated convergence theorem yields
\begin{align}
\lim_{q\to 1-}\sum_{k=1}^{\infty}f(k)\widetilde{F}^A(q)^k=\sum_{k=1}^{\infty}f(k).
\end{align}
Using (\ref{a-definition}) and letting $q\to 1-$ in (\ref{lem2.2-eq1}) and (\ref{lem2.2-eq2}), we obtain
\begin{align}
  \lim_{q\to 1-}
G_0(q)\P_x\Big[f(L_T^{A\circlearrowleft})\Big]=\left(a(x)+\frac{1}{p_A}\right)f(0)+\frac{1}{p_A}\sum_{k=1}^{\infty}f(k)\qquad \text{for}\ x\in \Z^2.
\end{align}
Multiplying both sides by $p_A$ proves the assertion.
\end{proof}

For $x\neq 0$ and $v=v_i$, let $r^A(x,v_i)$ denote the probability that the excursion currently in progress will be completed as an $A$-excursion. We set $r^A(x,v)=0$ when $x=0$ or $v=0$. More precisely, we define
    \begin{align}
      r^A(x,v)&:=\begin{cases}
       \P_x\Big(S_{\tau-1}=v_j\ \text{for\ some }j\ \text{with}\ (i,j)\in A\Big)&\text{if}\ x\neq0,\ v=v_i,\\[10pt]
        0&\text{if}\ x=0\ \text{or}\ v=0.
      \end{cases}
\end{align}
As in the derivation of (\ref{pij-identity}), this quantity can be expressed in terms of the potential kernel $a$:
\begin{align}
r^A(x,v_i)=\frac{1}{4}\sum_{\substack{1\le j\le 4\\(i,j)\in A}}\Big(a(x)+1-a(x-v_j)\Big)\qquad \text{for}\ x\neq 0.
\end{align}
For $(x,v)\in E$ and $l\in \Z_{\ge 0}$, we define
\begin{align}
\label{Def-h_f^A}
h_f^A(x,l,v)&:=\Big(p_Aa(x)-r^A(x,v)\Big)f(l)+\sum_{k=0}^\infty f(k+l).
\end{align}
Since $a(x)\ge 0$ and $0\le r^A(x,v)\le 1$, the function $h_f^A$ is nonnegative. Assumption (\ref{assumption-f}) implies that $h_f^A(x,0,0)>0$ for every $x\in \Z^2$. Moreover, (\ref{lem2.2-eq3}) can be rewritten as
\begin{align}
  \label{f(L_T^A)-limit-zero-h_f^A}
   \lim_{q\to 1-}p_AG_0(q)\P_x\Big[f(L_T^{A\circlearrowleft})\Big]=h_f^A(x,0,0)\qquad \text{for}\ x\in \Z^2.
\end{align}
This gives (\ref{PxvfLTA-lim}) for $v=0$ and $l=0$; we next extend (\ref{f(L_T^A)-limit-zero-h_f^A}) to arbitrary $(x,v)\in E$. The general case follows by applying the same result to $f(\cdot+l)$.

\begin{lem}
\label{lem2.3}
We have
 \begin{align}
 \label{lem2.3-eq1}
   \lim_{q\to 1-}p_AG_0(q)\P_{x,v}\Big[f(L_T^{A\circlearrowleft})\Big]&=h_f^A(x,0,v)\qquad \text{for}\ (x,v)\in E.
 \end{align}
\end{lem}
\begin{proof}
Since the case $v=0$ {or $x = 0$} has already been established, it remains to consider $x\neq 0$ and $v=v_i$ for some fixed $i\in\{1,2,3,4\}$. We have
\begin{align*}
\P_{x,v_i}\Big[f(L_T^{A\circlearrowleft})\Big]&=\sum_{m=1}^\infty\P_{x,v_i}\Big[f(L_m^{A\circlearrowleft})\Big]pq^{m-1}\\
&=f(0)\sum_{m=1}^\infty\P_{x}(m< \tau)pq^{m-1}+\sum_{m=1}^\infty\P_{x,v_i}\Big[f(L_m^{A\circlearrowleft}),\ \tau\le m\Big]pq^{m-1}.
    \stepcounter{equation}\tag{\theequation}
\end{align*}
By the definition of $F_x$, we have
\begin{align*}
\sum_{m=1}^\infty\P_{x}(m< \tau)pq^{m-1}&=\sum_{m=1}^\infty\sum_{k=m+1}^{\infty}\P_{x}(\tau=k)pq^{m-1}=\sum_{k=2}^\infty\sum_{m=1}^{k-1} \P_{x}(\tau=k)pq^{m-1}\\
&=\sum_{k=1}^\infty\P_{x}(\tau=k)(1-q^{k-1})=1-\frac{F_x(q)}{q}.
    \stepcounter{equation}\tag{\theequation}
\end{align*}
By the strong Markov property at $\tau$, we have
\begin{align*}
&\sum_{m=1}^\infty\P_{x,v_i}\Big[f(L_m^{A\circlearrowleft}),\ \tau\le m\Big]pq^{m-1}\\
&\qquad =\sum_{m=1}^\infty\sum_{j=1}^4\P_{x,v_i}\Big[f(L_m^{A\circlearrowleft}),\ \tau\le m,\ S_{\tau-1}=v_j\Big]pq^{m-1}\\
&\qquad =\sum_{m=1}^\infty\sum_{k=1}^m\sum_{j=1}^4\P_{x}\Big[{\P}_{0}\Big[f({L}_{m-k}^{A\circlearrowleft}+1_{\{(i,j)\in A\}})\Big],\ \tau=k,\ S_{\tau-1}=v_j\Big]pq^{m-1}\\
&\qquad =\sum_{k=1}^\infty\sum_{m=k}^\infty\sum_{j=1}^4\P_{x}\Big[{\P}_{0}\Big[f({L}_{m-k}^{A\circlearrowleft}+1_{\{(i,j)\in A\}})\Big],\ \tau=k,\ S_{\tau-1}=v_j\Big]pq^{m-1}\\
&\qquad =\sum_{k=1}^\infty\sum_{m=k+1}^\infty\sum_{j=1}^4\P_{x}\Big[{\P}_{0}\Big[f({L}_{m-k}^{A\circlearrowleft}+1_{\{(i,j)\in A\}})\Big],\ \tau=k,\ S_{\tau-1}=v_j\Big]pq^{m-1}\\
&\qquad\qquad +\sum_{k=1}^\infty\sum_{j=1}^4\P_{x}\Big[{\P}_{0}\Big[f(1_{\{(i,j)\in A\}})\Big],\ \tau=k,\ S_{\tau-1}=v_j\Big]pq^{k-1}\\
&\qquad =\sum_{k=1}^\infty\sum_{j=1}^4\P_{x}\Big[{\P}_{0}\Big[f({L}_{T}^{A\circlearrowleft}+1_{\{(i,j)\in A\}})\Big],\ \tau=k,\ S_{\tau-1}=v_j\Big]q^{k}\\
&\qquad\qquad +\sum_{k=1}^\infty\sum_{j=1}^4\P_{x}\Big[f(1_{\{(i,j)\in A\}}),\ \tau=k,\ S_{\tau-1}=v_j\Big]pq^{k-1}\\
&\qquad =\sum_{j=1}^4\P_0\Big[f({L}_{T}^{A\circlearrowleft}+1_{\{(i,j)\in A\}})\Big]\P_x\Big[q^\tau,\ S_{\tau-1}=v_j\Big]\\
&\qquad \qquad +p\sum_{j=1}^4f(1_{\{(i,j)\in A\}})\P_x\Big[q^{\tau-1},\ S_{\tau-1}=v_j\Big].
    \stepcounter{equation}\tag{\theequation}
\end{align*}
Therefore, by (\ref{lem2.2-eq3}), (\ref{GxFx-identity}), and (\ref{G_0(q)sim}), we have
\begin{align*}
&\lim_{q\to 1-}p_AG_0(q)\P_{x,v_i}\Big[f(L_T^{A\circlearrowleft})\Big]\\
&\qquad =\lim_{q\to 1-}p_AG_0(q)\left(\left(1-\frac{F_x(q)}{q}\right)f(0)+\sum_{j=1}^4\P_0\Big[f({L}_{T}^{A\circlearrowleft}+1_{\{(i,j)\in A\}})\Big]\P_x\Big[q^\tau,\ S_{\tau-1}=v_j\Big]\right.\\
&\qquad \qquad \left.+p\sum_{j=1}^4f(1_{\{(i,j)\in A\}})\P_x\Big[q^{\tau-1},\ S_{\tau-1}=v_j\Big]\right)\\
&\qquad =p_Aa(x)f(0)+\sum_{k=0}^\infty \sum_{j=1}^4f(k+1_{\{(i,j)\in A\}})\P_x(S_{\tau-1}=v_j)\\
&\qquad =p_Aa(x)f(0)+\sum_{k=0}^\infty f(k)-f(0)\sum_{\substack{1\le j\le 4\\(i,j)\in A}}\P_x(S_{\tau-1}=v_j)\\
&\qquad =\Big(p_Aa(x)-r^A(x,v_i)\Big)f(0)+\sum_{k=0}^\infty f(k).
\stepcounter{equation}\tag{\theequation}
\end{align*}
This completes the proof.
\end{proof}

This lemma implies the harmonicity of $h_f^A$.

\begin{lem}
\label{h-martingale-lem}
The function $h_f^A$ is harmonic for the Markov chain $(S,L^{A\circlearrowleft},V)$; that is, 
 \begin{align}
 \label{lem2.4-eq1}
\P_{x,v}\Big[h_f^A(S_{1},L_{1}^{A\circlearrowleft}+l,V_{1})\Big]=h_f^A(x,l,v)\qquad \text{for}\ (x,v)\in E,\ l\in \Z_{\ge 0}.
 \end{align}
In particular, the process $(h_f^A(S_n,L_n^{A\circlearrowleft},V_n))_{n\ge 0}$ is a nonnegative $\P_{x,v}$-martingale.
\end{lem}
\begin{proof}
By the Markov property, we have
\begin{align*}
\label{lem2.4-eq2}
&p_AG_0(q)\P_{x,v}\Big[f(L_T^{A\circlearrowleft}+l)\Big]\\
&\qquad =pp_AG_0(q)\P_{x,v}\Big[f(L_1^{A\circlearrowleft}+l)\Big]+p_AG_0(q)\sum_{m=2}^\infty \P_{x,v}\Big[f(L_m^{A\circlearrowleft}+l)\Big]pq^{m-1}\\
&\qquad =pp_AG_0(q)\P_{x,v}\Big[f(L_1^{A\circlearrowleft}+l)\Big]+q\P_{x,v}\Big[p_AG_0(q){\P}_{S_1,V_1}\Big[f({L}_T^{A\circlearrowleft}+l+l')\Big]_{l'=L_1^{A\circlearrowleft}}\Big].
 \stepcounter{equation}\tag{\theequation}
\end{align*}
For each fixed initial state $(x,v)$, the random vector $(S_1,L_1^{A\circlearrowleft},V_1)$ takes only finitely many values. Therefore, letting $q\to 1-$ in (\ref{lem2.4-eq2}) and using (\ref{G_0(q)sim}) and (\ref{lem2.3-eq1}), we obtain (\ref{lem2.4-eq1}).

The martingale property follows immediately from the harmonicity of $h_f^A$ and the Markov property of $(S,L^{A\circlearrowleft},V)$. This completes the proof.
\end{proof}

Finally, we prove the penalization limit.

\begin{prop}
\label{thm2.5}
We have
 \begin{align}
 \label{thm2.5-eq1}
   \lim_{q\to 1-}\frac{\P_x[1_{\Lambda_n}\cdot f(L_{T}^{A\circlearrowleft})]}{\P_x[f(L_{T}^{A\circlearrowleft})]}=\P_x\left[1_{\Lambda_n}\cdot \frac{h_f^A(S_n,L_n^{A\circlearrowleft},V_n)}{h_f^{A}(x,0,0)}\right]
 \end{align}
 for $n\ge 0,\ x\in \Z^2,\ \Lambda_n\in \F_n$.
\end{prop}
\begin{proof}
We have
\begin{align*}
\P_x\Big[1_{\Lambda_n}\cdot f(L_{T}^{A\circlearrowleft})\Big]&=\sum_{m=1}^n\P_x\Big[1_{\Lambda_n}\cdot f(L_{m}^{A\circlearrowleft})\Big]pq^{m-1}+\sum_{m=n+1}^\infty\P_x\Big[1_{\Lambda_n}\cdot f(L_{m}^{A\circlearrowleft})\Big]pq^{m-1}.
   \stepcounter{equation}\tag{\theequation}
\end{align*}
By (\ref{G_0(q)sim}), we have
\begin{align*}
&p_AG_0(q)\left|\sum_{m=1}^n\P_x\Big[1_{\Lambda_n}\cdot f(L_{m}^{A\circlearrowleft})\Big]pq^{m-1}\right|\\
&\qquad \le p_A \max_{0\le m\le n}f(m)\sum_{m=1}^n q^{m-1}\cdot (1-q)G_0(q)\to 0\qquad \text{as}\ q\to 1-.
\stepcounter{equation}\tag{\theequation}
\end{align*}
By the Markov property and (\ref{lem2.3-eq1}) applied to the shifted function $f(\cdot+l)$, we have
\begin{align*}
&p_AG_0(q)\sum_{m=n+1}^\infty\P_x\Big[1_{\Lambda_n}\cdot f(L_{m}^{A\circlearrowleft})\Big]pq^{m-1}\\
&\qquad =p_AG_0(q)\sum_{m=n+1}^\infty\P_x\Big[1_{\Lambda_n}\cdot \P_{S_n,V_n}\Big[f(L_{m-n}^{A\circlearrowleft}+l)\Big]_{l=L_n^{A\circlearrowleft}}\Big]pq^{m-1}\\
&\qquad =\P_x\Big[1_{\Lambda_n}\cdot p_AG_0(q)\P_{S_n,V_n}\Big[f(L_{T}^{A\circlearrowleft}+l)\Big]_{l=L_n^{A\circlearrowleft}}\Big]q^{n}\\
&\qquad \to \P_x\Big[1_{\Lambda_n}\cdot h_f^A(S_n,L_n^{A\circlearrowleft},V_n)\Big]
 \stepcounter{equation}\tag{\theequation}
\end{align*}
as $q\to 1-$. For each fixed $n$, the random vector $(S_n,L_n^{A\circlearrowleft},V_n)$ takes only finitely many values, which justifies passing to the limit inside the expectation. Thus, we have
\begin{align}
\label{thm2.5-eq2}
  \lim_{q\to 1-}p_AG_0(q)\P_x\Big[1_{\Lambda_n}\cdot f(L_{T}^{A\circlearrowleft})\Big]=\P_x\Big[1_{\Lambda_n}\cdot h_f^A(S_n,L_n^{A\circlearrowleft},V_n)\Big].
\end{align}
Combining (\ref{thm2.5-eq2}) with (\ref{f(L_T^A)-limit-zero-h_f^A}), we obtain (\ref{thm2.5-eq1}). This completes the proof.
\end{proof}

This completes the proof of all the assertions in Theorem \ref{thm1.1}.

\subsection{Penalized Process}
\label{S2.2}
In this subsection, we derive a path decomposition formula for the penalized process. 

For $x\in\Z^2$, we denote by $\Q_x^{f,A\circlearrowleft}$ the penalized measure on $\F_\infty$ arising in the limit (\ref{penal-limit}), that is,
\begin{align}
\label{def-QfA}
  \frac{d\Q_x^{f,A\circlearrowleft}}{d\P_x}\Big|_{\F_n}=\frac{h_f^A(S_n,L_n^{A\circlearrowleft},V_n)}{h_f^A(x,0,0)}\qquad \text{for}\ n\ge 0,\ x\in \Z^2.
\end{align}
For the existence of $\Q_x^{f,A\circlearrowleft}$, we refer the reader, for instance, to Theorem A.1 of \cite{Yano-Penal}. By Lemma \ref{h-martingale-lem}, $h_f^A$ is harmonic for the process $(S,L^{A\circlearrowleft},V)$. Therefore, under the probability measure $\Q_x^{f,A\circlearrowleft}$, the process is the h-transform of $(S,L^{A\circlearrowleft},V)$ with respect to $h_f^A$. Since $(S,L^{A\circlearrowleft},V)$ is a strong Markov process, it follows from the general theory of h-transforms that $(S,L^{A\circlearrowleft},V)$ under $\Q_x^{f,A\circlearrowleft}$ is also a strong Markov process (see, e.g., Theorem 11.9 of \cite{C-W}). In general, however, the coordinate process $S$ alone need not be Markovian.

For $k\ge 0$, we define the following process:
\begin{align}
  Z_n^{(k),A}:=1_{\{L_n^{A\circlearrowleft}<k\}}+\Big(1+p_Aa(S_n)-r^A(S_n,V_n)\Big)1_{\{L_n^{A\circlearrowleft}=k\}}\qquad \text{for}\ n\ge 0.
\end{align}
Then we have
\begin{align}
\label{h_f^A-Z}
h_f^A(S_n,L_n^{A\circlearrowleft},V_n)=\sum_{k=0}^\infty f(k)Z_n^{(k),A}.
\end{align}
Moreover, if we set $f^{(k)}(l):=1_{\{l=k\}}$ for $k\ge 0$, then $h_{f^{(k)}}^A(S_n,L_n^{A\circlearrowleft},V_n)=Z_n^{(k),A}$. Hence, by Lemma \ref{h-martingale-lem}, the process $(Z_n^{(k),A})_{n\ge 0}$ is a nonnegative $\P_x$-martingale. In particular, by the martingale property, we have
\begin{align}
  \P_x[Z_n^{(k),A}]=\P_x[Z_0^{(k),A}]=\begin{cases}
    1+p_Aa(x)&\text{for}\ k=0,\\[10pt]
    1&\text{for}\ k\ge 1.
  \end{cases}
\end{align}
For each $k\ge 0$, we can therefore define a probability measure $\Q_x^{(k),A\circlearrowleft}$ on $\F_\infty$ by
\begin{align}
\label{def-Q0A}
 \frac{d\Q_x^{(0),A\circlearrowleft}}{d\P_x}\Big|_{\F_n}&=\frac{Z_n^{(0),A}}{1+p_Aa(x)},\\
  \frac{d\Q_x^{(k),A\circlearrowleft}}{d\P_x}\Big|_{\F_n}&=Z_n^{(k),A}\qquad \text{for}\ k\ge 1.
\end{align}
Note that these measures do not depend on $f$.

The decomposition in (\ref{h_f^A-Z}), together with the definitions of $\Q_x^{f,A\circlearrowleft}$ and $\Q_x^{(k),A\circlearrowleft}$, gives the following mixture representation.

\begin{lem}
\label{lem-2.6}
For $x\in \Z^2$, we have
\begin{align}
\label{lem2.6-eq1}
\Q_x^{f,A\circlearrowleft}=\sum_{k=0}^\infty w_f^{(k),A}(x)\Q_x^{(k),A\circlearrowleft},
\end{align}
where
\begin{align}
  w_f^{(0),A}(x):=\frac{(1+p_Aa(x))f(0)}{h_f^A(x,0,0)},\qquad w_f^{(k),A}(x):=\frac{f(k)}{h_f^A(x,0,0)}\ \text{for}\ k\ge 1.
\end{align}
\end{lem}

The next lemma shows that, for each $k\ge 0$, the walk completes exactly $k$ $A$-excursions, $\Q_x^{(k),A\circlearrowleft}$-almost surely.

\begin{lem}
\label{lem2.7}
Let $L_\infty^{A\circlearrowleft}:=\lim_{n\to \infty} L_n^{A\circlearrowleft}$. Then
 \begin{align}
 \label{lem2.7-eq1}
   \Q_x^{(k),A\circlearrowleft}(L_\infty^{A\circlearrowleft}=k)=1\qquad \text{for}\ k\ge 0,\ x\in \Z^2.
 \end{align}
\end{lem}
\begin{proof}
First, suppose that $k\ge 1$. By recurrence, the walk reaches the origin and subsequently completes infinitely many excursions, $\P_x$-almost surely. By the strong Markov property, the types of these excursions are independent and identically distributed, and each excursion is an $A$-excursion with probability $p_A>0$. Hence, we have
\begin{align}
  L_\infty^{A\circlearrowleft}=\infty,
\qquad \P_x\text{-a.s.}
\end{align}
It follows that
 \begin{align*}
 \Q_x^{(k),A\circlearrowleft}(L_n^{A\circlearrowleft}<k)=\P_x\Big[1_{\{L_n^{A\circlearrowleft}<k\}}\cdot Z_n^{(k),A}\Big]=\P_x(L_n^{A\circlearrowleft}<k)\to 0\qquad \text{as}\ n\to \infty.
 \stepcounter{equation}\tag{\theequation}  
 \end{align*}
Thus, we have
 \begin{align}
   \Q_x^{(k),A\circlearrowleft}(L_\infty^{A\circlearrowleft}<k)=0.
 \end{align}
On the other hand, we have
  \begin{align*}
 \Q_x^{(k),A\circlearrowleft}(L_n^{A\circlearrowleft}\ge k+1)&=\P_x\Big[1_{\{L_n^{A\circlearrowleft}\ge k+1\}}\cdot Z_n^{(k),A}\Big]= 0 \qquad \text{for}\ n\ge 0.
 \stepcounter{equation}\tag{\theequation}  
 \end{align*}
Together, these observations prove the assertion for $k\ge 1$.

Next, we consider the case $k=0$. Since
 \begin{align*}
  \Q_x^{(0),A\circlearrowleft}(L_n^{A\circlearrowleft}\ge 1)&=\frac{1}{1+p_Aa(x)}\P_x\Big[1_{\{L_n^{A\circlearrowleft}\ge 1\}}\cdot Z_n^{(0),A}\Big]= 0\qquad \text{for}\ n\ge 0,
   \stepcounter{equation}\tag{\theequation}  
 \end{align*}
we obtain the assertion. This completes the proof.
\end{proof}

Lemmas \ref{lem-2.6} and \ref{lem2.7} identify the distribution of the total number of $A$-excursions under $\Q_x^{f,A\circlearrowleft}$.

\begin{cor}
\label{cor-2.8}
For $k\ge 0$, we have
\begin{align}
\label{cor2.8-eq1}
  \Q_x^{f,A\circlearrowleft}(  L_\infty^{A\circlearrowleft}=k)=w_f^{(k),A}(x)\qquad \text{for}\ x\in \Z^2.
\end{align}
In particular, we have
\begin{align}
   \Q_x^{f,A\circlearrowleft}(L_\infty^{A\circlearrowleft}<\infty)=1.
 \end{align}
Moreover, for $k\ge 0$ and $x\in \Z^2$, we have
\begin{align}
\label{cor2.8-eq2}
  \Q_x^{f,A\circlearrowleft}\Big[F,\   L_\infty^{A\circlearrowleft}=k\Big]=w_f^{(k),A}(x)\Q_x^{(k),A\circlearrowleft}[F],
\end{align}
where $F$ denotes a bounded $\F_\infty$-measurable random variable.
\end{cor}

We define the following function:
\begin{align}
\widetilde{h}^A(x,v):=1+p_Aa(x)-r^A(x,v)\qquad \text{for}\ (x,v)\in E.
\end{align}
This function is strictly positive on $E$. Indeed, $\widetilde{h}^A(0,0)=1$, while $\widetilde{h}^A(x,v)\ge p_Aa(x)>0$ whenever $x\neq 0$. We have
\begin{align}
    Z_n^{(k),A}=1_{\{L_n^{A\circlearrowleft}<k\}}+\widetilde{h}^A(S_n,V_n)1_{\{L_n^{A\circlearrowleft}=k\}}\qquad \text{for}\ k\ge 0,\ n\ge 0.
\end{align}
Note that
\begin{align}
  \widetilde{h}^A(x,0)=1+p_Aa(x)\qquad \text{for}\ x\in \Z^2.
\end{align}
By Lemma \ref{h-martingale-lem}, applied to $f^{(0)}(k)=1_{\{k=0\}}$, the process $(\widetilde{h}^A(S_n,V_n)1_{\{L_n^{A\circlearrowleft}=0\}})_{n\ge0}$ is a nonnegative $\P_{x,v}$-martingale for every $(x,v)\in E$, with initial value $\widetilde{h}^A(x,v)>0$. We therefore define a probability measure $\P_{x,v}^{A\updownarrow}$ on $\F_\infty$ by
\begin{align}
\label{def-P_{xv}^Aupdown}
  \frac{d\P_{x,v}^{A\updownarrow}}{d\P_{x,v}}\Big|_{\F_n}:=\frac{\widetilde{h}^A(S_n,V_n)}{\widetilde{h}^A(x,v)}1_{\{L_n^{A\circlearrowleft}=0\}}.
\end{align}
Note that the probability measure $\P_{x,v}^{A\updownarrow}$ does not depend on $f$. Under $\P_{x,v}^{A\updownarrow}$, the process $(S,V)$ is the $\widetilde{h}^A$-transform of the chain killed when an $A$-excursion is completed, and is therefore a strong Markov chain. By definition, it is clear that
\begin{align}
   \P_{x,v}^{A\updownarrow}(L_\infty^{A\circlearrowleft}=0)=1.
\end{align}
Note that
\begin{align}
\label{Q0P0}
\Q_x^{(0),A\circlearrowleft}=\P_{x,0}^{A\updownarrow}\qquad \text{for}\ x\in \Z^2.
\end{align}

Under $\Q_x^{f,A\circlearrowleft}$, conditional on $L_\infty^{A\circlearrowleft}=k\ge1$, provided that this event has positive probability, the path up to $\sigma_k$ has the same law as the original simple random walk started at $x$ and stopped at $\sigma_k$. The post-$\sigma_k$ process is independent of this stopped path and has law $\P_{0,0}^{A\updownarrow}$. Conditional on $L_\infty^{A\circlearrowleft}=0$, the entire path has law $\P_{x,0}^{A\updownarrow}$. More precisely, we have the following proposition.

\begin{prop}
\label{thm2.9}
For $x\in \Z^2$, we have
 \begin{align}
 \label{thm2.9-eq2}
   \Q_x^{f,A\circlearrowleft}\Big[F\cdot G\circ \theta_{\sigma_k},\ L_\infty^{A\circlearrowleft}=k\Big]&=w_f^{(k),A}(x)\P_x[F]\cdot\P_{0,0}^{A\updownarrow}[G]\qquad \text{for}\ k\ge 1,\\
   \label{thm2.9-eq3}
     \Q_x^{f,A\circlearrowleft}\Big[G,\ L_\infty^{A\circlearrowleft}=0\Big]&=w_f^{(0),A}(x)\P_{x,0}^{A\updownarrow}[G],
 \end{align}
where $F$ denotes a bounded $\F_{\sigma_k}$-measurable random variable, $G$ denotes a bounded $\F_\infty$-measurable random variable, and $\theta$ denotes the shift operator.
\end{prop}
\begin{proof}
First, we consider the case $k\ge 1$. For any $m\ge 0$ and a bounded $\F_m$-measurable random variable $G_m$, we have 
\begin{align*}
\Q_x^{(k),A\circlearrowleft}\Big[F\cdot G_m\circ \theta_{\sigma_k}\Big]&=\sum_{n=1}^\infty \Q_x^{(k),A\circlearrowleft}\Big[F\cdot G_m\circ \theta_{n},\ \sigma_k=n\Big]\\
&=\sum_{n=1}^\infty \P_x\Big[F1_{\{\sigma_k=n\}}\cdot G_m\circ \theta_{n}\cdot Z_{n+m}^{(k),A}\Big]\\
&=\sum_{n=1}^\infty \P_x\Big[F1_{\{\sigma_k=n\}}\cdot G_m\circ \theta_{n}\cdot \Big(\widetilde{h}^A(S_m,V_m)1_{\{L_m^{A\circlearrowleft}=0\}}\Big)\circ \theta_n\Big]\\
&=\sum_{n=1}^\infty \P_x\Big[F,\ \sigma_k=n\Big]\cdot\P_0\Big[G_m\cdot \widetilde{h}^A(S_m,V_m)1_{\{L_m^{A\circlearrowleft}=0\}}\Big]\\
&=\P_x[F]\cdot \P_{0,0}^{A\updownarrow}[G_m].
   \stepcounter{equation}\tag{\theequation}  
\end{align*}
Thus, by (\ref{cor2.8-eq2}), we have
\begin{align}
\label{thm2.9-eq1}
   \Q_x^{f,A\circlearrowleft}\Big[F\cdot G_m\circ \theta_{\sigma_k},\ L_\infty^{A\circlearrowleft}=k\Big]=w_f^{(k),A}(x)\P_x[F]\cdot \P_{0,0}^{A\updownarrow}[G_m].
\end{align}
A monotone class argument extends (\ref{thm2.9-eq1}) to every bounded $\F_\infty$-measurable random variable $G$, proving (\ref{thm2.9-eq2}).

Finally, (\ref{thm2.9-eq3}) follows directly from Corollary \ref{cor-2.8} and (\ref{Q0P0}). This completes the proof.
\end{proof}

Recall that for $x\neq 0$, the law $\P_x^{\updownarrow}$ of the two-dimensional simple random walk conditioned to avoid the origin is defined as follows:
\begin{align}
  \frac{d\P_x^{\updownarrow}}{d\P_x}\Big|_{\F_n}=\frac{a(S_n)}{a(x)}1_{\{n<\tau\}}.
\end{align}

Let $A^c:=\{1,2,3,4\}^2\setminus A$. Suppose that $A\neq\{1,2,3,4\}^2$. Under $\P_{0,0}^{A\updownarrow}$, conditional on $\{L_\infty=n\}$, the walk completes $n$ independent excursions, each having the law of an ordinary excursion conditioned to be an $A^c$-excursion. It then moves from the origin to a uniformly chosen neighboring site $v_i$, independently of these excursions, and subsequently follows $\P_{v_i}^{\updownarrow}$. More precisely, we have the following proposition.

\begin{prop}
\label{thm2.10}
Assume that $A\neq\{1,2,3,4\}^2$. For $n\ge 0$ and $i\in \{1,2,3,4\}$, we have
   \begin{align*}
   \label{Thm2.10-eq1}
   &\P_{0,0}^{A\updownarrow}\left[\left(\prod_{k=1}^n G_k\circ\theta_{\rho_{k-1}}\right)\cdot H\circ \theta_{\rho_{n}+1},\ L_\infty=n,\ S_{\rho_n+1}=v_i\right]\\
   &\qquad =\frac{p_A}{4}\left(\prod_{k=1}^n\P_0\Big[G_k,\ (I_1,J_1)\notin A\Big]\right)\cdot \P_{v_i}^{\updownarrow}[H],
      \stepcounter{equation}\tag{\theequation}   
 \end{align*}
where $G_1,...,G_n$ denote bounded $\F_\tau$-measurable random variables and $H$ denotes a bounded $\F_\infty$-measurable random variable. For $n=0$, both products are interpreted as $1$. In particular, we have
  \begin{align}
    \P_{0,0}^{A\updownarrow}(L_\infty=n)=p_A(1-p_A)^{n}\qquad \text{for}\ n\ge 0.
 \end{align}
\end{prop}
\begin{proof}
Let $G$ be a bounded $\F_\tau$-measurable random variable. For each $m\ge 0$, applying the optional sampling theorem at the bounded stopping time $\tau\wedge m$ gives
  \begin{align*}
\P_{0,0}^{A\updownarrow}\Big[G_,\ \tau\le m\Big]&=\P_{0}\Big[G1_{\{\tau\le m\}}\cdot \widetilde{h}^A(S_m,V_m)1_{\{L_m^{A\circlearrowleft}=0\}}\Big]\\
&=\P_{0}\Big[G1_{\{\tau\le m\}}\cdot \widetilde{h}^A(S_\tau,V_\tau)1_{\{L_\tau^{A\circlearrowleft}=0\}}\Big]\\
&=\P_{0}\Big[G1_{\{\tau\le m\}}\cdot1_{\{(I_1,J_1)\notin A\}}\Big].
  \stepcounter{equation}\tag{\theequation}   
\end{align*}
Letting $m\to \infty$, we have
\begin{align}
\label{thm2.10-eq1}
  \P_{0,0}^{A\updownarrow}\Big[G,\ \tau<\infty\Big]=\P_{0}\Big[G,\ (I_1,J_1)\notin A\Big].
\end{align}

For $k\ge 0$, let $H_k$ be a bounded $\F_k$-measurable random variable. For $m\ge k+1$, we have
\begin{align*}
\label{Thm-1.14-eq1}
&\P_{0,0}^{A\updownarrow}\Big[H_k\circ \theta_1,\ \tau> m,\ S_1=v_i\Big]\\
&\qquad =\P_{0}\Big[H_k\circ \theta_1\cdot \widetilde{h}^A(S_m,V_m)1_{\{L_m^{A\circlearrowleft}=0\}},\ \tau> m,\ S_1=v_i\Big]\\
&\qquad =\frac{1}{4}\P_{v_i}\Big[H_k\cdot \widetilde{h}^A(S_{m-1},v_i)1_{\{L_{m-1}=0\}}\Big]\\
&\qquad =\frac{p_A}{4}\P_{v_i}\Big[H_k\cdot a(S_{m-1})1_{\{L_{m-1}=0\}}\Big]+\frac{1}{4}\P_{v_i}\Big[H_k\cdot\Big(1-r^A(S_{m-1},v_i)\Big)1_{\{L_{m-1}=0\}}\Big].
  \stepcounter{equation}\tag{\theequation}   
\end{align*}
By the optional sampling theorem and (\ref{avi=1}), the first term on the right-hand side of (\ref{Thm-1.14-eq1}) can be rewritten as
\begin{align*}
  \frac{p_A}{4}\P_{v_i}\Big[H_k\cdot a(S_{m-1})1_{\{L_{m-1}=0\}}\Big]=\frac{p_Aa(v_i)}{4}\P_{v_i}\left[H_k\cdot \frac{a(S_{k})}{a(v_i)}1_{\{L_{k}=0\}}\right]=\frac{p_A}{4}\P_{v_i}^{\updownarrow}[H_k].
    \stepcounter{equation}\tag{\theequation}   
\end{align*}
For the second term on the right-hand side of (\ref{Thm-1.14-eq1}), the bound $0\le r^A \le 1$ gives
\begin{align}
  \left|\frac{1}{4}\P_{v_i}\Big[H_k\cdot\Big(1-r^A(S_{m-1},v_i)\Big)1_{\{L_{m-1}=0\}}\Big]\right|\le \frac{\|H_k\|_\infty}{4}\P_{v_i}(L_{m-1}=0)\to 0\qquad \text{as}\ m\to \infty.
\end{align}
Thus, letting $m\to\infty$ in (\ref{Thm-1.14-eq1}), we have
\begin{align*}
\label{thm2.10-eq2}
&\P_{0,0}^{A\updownarrow}\Big[H_k\circ \theta_1,\ \tau=\infty ,\ S_1=v_i\Big]=\frac{p_A}{4}\P_{v_i}^{\updownarrow}[H_k].
  \stepcounter{equation}\tag{\theequation}   
\end{align*}
A monotone class argument extends (\ref{thm2.10-eq2}) to every bounded $\F_\infty$-measurable random variable $H$. For $n\ge 1$, applying the strong Markov property successively at $\rho_1,...,\rho_n$ and using (\ref{thm2.10-eq1}) and (\ref{thm2.10-eq2}), we obtain
\begin{align*}
&\P_{0,0}^{A\updownarrow}\Big[G_1\cdot G_2\circ \theta_{\rho_1}\cdots G_n\circ \theta_{\rho_{n-1}}\cdot H\circ \theta_{\rho_{n}+1},\ L_\infty=n,\ S_{\rho_n+1}=v_i\Big]\\
&\qquad =\prod_{k=1}^n \P_{0,0}^{A\updownarrow}\Big[G_k,\ \tau<\infty\Big]\cdot \P_{0,0}^{A\updownarrow}\Big[H\circ \theta_1,\ \tau=\infty,\ S_1=v_i\Big]\\
&\qquad =\prod_{k=1}^n\P_0\Big[G_k,\ (I_1,J_1)\notin A\Big]\cdot \frac{p_A}{4}\P_{v_i}^{\updownarrow}[H].
    \stepcounter{equation}\tag{\theequation}  
\end{align*}
The case $n=0$ follows directly from (\ref{thm2.10-eq2}). This completes the proof.
\end{proof}

We next consider the case $A=A_\ast:=\{1,2,3,4\}^2$. Under $\P_{0,0}^{A\updownarrow}$, the walk never returns to the origin after time $0$. Its first step is to a neighboring site chosen uniformly from $\{v_1,v_2,v_3,v_4\}$. Conditional on this site being $v_i$, the path from time $1$ onward has law $\P_{v_i}^{\updownarrow}$. More precisely, we have the following proposition.

\begin{prop}
\label{thm2.11}
Assume that $A=A_\ast$. Then we have
 \begin{align}
   \P_{0,0}^{A_\ast\updownarrow}\Big[H\circ \theta_1,\ S_1=v_i\Big]=\frac{1}{4}\P_{v_i}^{\updownarrow}[H]\qquad \text{for}\ i\in \{1,2,3,4\},
 \end{align}
where $H$ denotes a bounded $\F_\infty$-measurable random variable.
\end{prop}
\begin{proof}
In this case, $p_{A_\ast}=1$, $L^{A_\ast\circlearrowleft}=L$ under $\mathbb P_0$, and
\begin{align}
  \widetilde{h}^{A_\ast}(x,v_i)=a(x),
\qquad \text{for}\ x\neq 0,\ i\in \{1,2,3,4\}.
\end{align}
Let $m\ge 0$, and let $H_m$ be a bounded $\F_m$-measurable random variable. We have
 \begin{align*}
 \P_{0,0}^{A_\ast\updownarrow}\Big[H_m\circ \theta_1,\ S_1=v_i\Big]&=\P_{0}\Big[H_m\circ \theta_1\cdot 1_{\{S_1=v_i\}}\cdot \widetilde{h}^{A_\ast}(S_{m+1},V_{m+1})1_{\{L_{m+1}^{A_\ast\circlearrowleft}=0\}}\Big]\\
 &=\P_0(S_1=v_i)\P_{v_i,v_i}\Big[H_m\cdot \widetilde{h}^{A_\ast}(S_m,V_m)1_{\{L_m=0\}}\Big]\\
 &=\frac{1}{4}\P_{v_i}\Big[H_m\cdot a(S_m)1_{\{L_m=0\}}\Big]\\
 &=\frac{1}{4}\P_{v_i}^{\updownarrow}[H_m].
   \stepcounter{equation}\tag{\theequation}  
 \end{align*}
A monotone class argument extends this identity to every bounded $\F_\infty$-measurable random variable $H$. This completes the proof.
\end{proof}

Finally, let $x\neq 0$. Under $\P_{x,0}^{A\updownarrow}$, conditional on $\tau<\infty$, the path up to $\tau$ has the same law as the original simple random walk started at $x$ and stopped upon hitting the origin. Under this conditional law, the post-$\tau$ process is independent of this stopped path and has law $\P_{0,0}^{A\updownarrow}$. Conditional on $\tau=\infty$, the entire path has law $\P_x^{\updownarrow}$. More precisely, we have the following proposition.

\begin{prop}
\label{thm2.12}
For $x\neq 0$, we have
\begin{align}
\label{thm2.12-eq2}
  \P_{x,0}^{A\updownarrow}\Big[G\cdot H\circ \theta_\tau,\ \tau<\infty\Big]&=\frac{1}{1+p_Aa(x)}\cdot \P_x[G]\cdot \P_{0,0}^{A\updownarrow}[H],\\[10pt]
  \P_{x,0}^{A\updownarrow}\Big[H,\ \tau=\infty\Big]&=\frac{p_Aa(x)}{1+p_Aa(x)}\cdot\P_x^\updownarrow[H],
\end{align}
where $G$ denotes a bounded $\F_\tau$-measurable random variable and $H$ denotes a bounded $\F_\infty$-measurable random variable.
\end{prop}
\begin{proof}
We first prove (\ref{thm2.12-eq2}). Since the initial path segment from $x\neq 0$ to the origin is not counted as an excursion, we have $L_\tau^{A\circlearrowleft}=0$ and $(S_\tau,V_\tau)=(0,0)$ on $\{\tau<\infty\}$. Thus, applying the optional sampling theorem at $\tau\wedge n$, we have
   \begin{align*}
    \P_{x,0}^{A\updownarrow}\Big[G,\ \tau\le n\Big]&=\frac{1}{\widetilde{h}^A(x,0)}\cdot\P_{x}\Big[G1_{\{\tau\le n\}}\cdot \widetilde{h}^A(S_n,V_n)1_{\{L_n^{A\circlearrowleft}=0\}}\Big]\\
    &=\frac{1}{1+p_Aa(x)}\cdot\P_{x}\Big[G1_{\{\tau\le n\}}\cdot \widetilde{h}^A(S_\tau,V_\tau)1_{\{L_\tau^{A\circlearrowleft}=0\}}\Big]\\
    &=\frac{1}{1+p_Aa(x)}\cdot\P_{x}\Big[G,\ \tau\le n\Big]
\stepcounter{equation}\tag{\theequation} 
  \end{align*}
for $n\ge 0$. Letting $n\to\infty$ and using recurrence and the bounded convergence theorem, we obtain
  \begin{align}
  \label{thm2.12-eq1}
    \P_{x,0}^{A\updownarrow}\Big[G,\ \tau<\infty\Big]=\frac{1}{1+p_Aa(x)}\cdot \P_{x}[G].
  \end{align}
The strong Markov property then gives
  \begin{align*}
  \P_{x,0}^{A\updownarrow}\Big[G\cdot H\circ \theta_\tau,\ \tau<\infty\Big]&=\P_{x,0}^{A\updownarrow}\Big[G,\ \tau<\infty\Big]\cdot\P_{0,0}^{A\updownarrow}[H]=\frac{1}{1+p_Aa(x)}\cdot \P_{x}[G]\cdot \P_{0,0}^{A\updownarrow}[H].
\stepcounter{equation}\tag{\theequation} 
  \end{align*}
This proves (\ref{thm2.12-eq2}).

Next, we prove the second assertion. For $n\ge 0$, let $H_n$ be a bounded $\F_n$-measurable random variable. By (\ref{thm2.12-eq1}), we have
\begin{align}
  \P_{x,0}^{A\updownarrow}(\tau=\infty)=1-\P_{x,0}^{A\updownarrow}(\tau<\infty)=\frac{p_Aa(x)}{1+p_Aa(x)}.
\end{align}
Thus, by the Markov property of $(S,V)$, we have
\begin{align*}
  \P_{x,0}^{A\updownarrow}\Big[H_n,\ \tau=\infty\Big]&=\P_{x,0}^{A\updownarrow}\Big[H_n1_{\{\tau> n\}}\cdot \P_{S_n,0}^{A\updownarrow}(\tau=\infty)\Big]\\
  &=\P_{x,0}^{A\updownarrow}\left[H_n1_{\{\tau> n\}}\cdot \frac{p_Aa(S_n)}{1+p_Aa(S_n)}\right]\\
    &=\P_{x}\left[H_n1_{\{\tau> n\}}\cdot \frac{p_Aa(S_n)}{1+p_Aa(S_n)}\cdot \frac{\widetilde{h}^A(S_n,0)}{\widetilde{h}^A(x,0)}1_{\{L_n^{A\circlearrowleft}=0\}}\right]\\
          &=\frac{p_Aa(x)}{1+p_Aa(x)}\cdot \P_{x}\left[H_n\cdot \frac{a(S_n)}{a(x)}1_{\{L_n=0\}}\right]\\
           &=\frac{p_Aa(x)}{1+p_Aa(x)}\cdot \P_{x}^{\updownarrow} \left[H_n\right].
\stepcounter{equation}\tag{\theequation} 
  \end{align*}
By the monotone class theorem, this identity extends to every bounded $\F_\infty$-measurable random variable $H$. This completes the proof.
\end{proof}

This completes the path decomposition of the penalized process under $\Q_x^{f,A\circlearrowleft}$.


\section{Penalization by Edge-Crossing Counts}
\label{S3}
In this section, we study the penalization problem with weights of the form $\Gamma=f(N^B)$
and derive a path decomposition for the penalized process.

\subsection{Penalization Limit}
\label{S3.1}
In this subsection, we prove Theorem \ref{thm1.2}.

Let $e_1,...,e_4$ denote the four unoriented edges incident to the origin, with $e_i:=\{0,v_i\}$. Fix a nonempty subset $B\subseteq \{e_1,e_2,e_3,e_4\}.$ We refer to the edges in $B$ as $B$-edges. Recall that $N^B=(N_n^B)_{n\ge0}$ counts the total number of crossings of $B$-edges:
\begin{align}
  N_0^B:=0,\qquad N_n^B:=\sum_{k=1}^n 1_{\{\{S_{k-1},S_k\}\in B\}}\ \text{for}\ n\ge 1 .
\end{align}
Set $\varsigma_0:=0$. For $k\ge 1$, let $\varsigma_k$ denote the time of the $k$-th crossing of a $B$-edge:
\begin{align}
\varsigma_k:=\inf\Big\{n>\varsigma_{k-1}:\ \{S_{n-1},S_n\}\in B\Big\}.
\end{align}
Let $\widetilde{B}$ denote the set of nonzero endpoints of the edges belonging to $B$:
\begin{align}
  \widetilde{B}:=\Big\{v\in \{v_1,v_2,v_3,v_4\}:\ \{0,v\}\in B\Big\}.
\end{align}
For $n\ge 0$ and $x\in \Z^2$, we define
  \begin{align}
   u_n(x)&:=\P_x\Big(S_n=0,\ \varsigma_1>n\Big),\qquad v_n(x):=\P_x\Big(S_n\in \widetilde{B},\ \varsigma_1>n\Big).
  \end{align}
Define the corresponding generating functions by
  \begin{align}
   U_x(\xi):=\sum_{n=0}^\infty u_n(x)\xi^n,\qquad V_x(\xi):=\sum_{n=0}^\infty v_n(x)\xi^n\qquad \text{for}\ 0\le \xi<1.
 \end{align}

The following lemma expresses $U_x$ and $V_x$ in terms of the functions $G_y$ defined in (\ref{def-Green}).

\begin{lem} \label{lem:UandV}
For $x\in \Z^2$ and $0\le \xi<1$, we have
\begin{align}
\label{Uxq-def}
U_x(\xi)&=\frac{1}{\Delta_B(\xi)}\left(\left(1+\frac{|B|\xi}{4}G_{v_1}(\xi)\right)G_x(\xi)-\frac{\xi}{4}G_0(\xi)\Sigma_{x,B}(\xi)\right),\\
\label{Vxq-def}
   V_x(\xi)&=\frac{1}{\Delta_B(\xi)}\left(\left(1+\frac{|B|\xi}{4}G_{v_1}(\xi)\right)\Sigma_{x,B}(\xi)-\frac{\xi}{4}\Sigma_{B,B}(\xi)G_x(\xi)\right),
\end{align}
where 
\begin{align}
\Sigma_{x,B}(\xi)&:=\sum_{v\in \widetilde{B}}G_{x-v}(\xi),\\
\Sigma_{B,B}(\xi)&:=\sum_{v,w\in \widetilde{B}}G_{v-w}(\xi),\\
\label{DeltaBq-def}
  \Delta_B(\xi)&:=\left(1+\frac{|B|\xi}{4}G_{v_1}(\xi)\right)^2-\frac{\xi}{4}G_0(\xi)\left(\frac{\xi}{4}\Sigma_{B,B}(\xi)\right).
\end{align}
\end{lem}
\begin{proof}
Decomposing according to the first crossing of a $B$-edge and using the Markov property and symmetry, we obtain
  \begin{align*}
  \label{lem3.1-eq1}
 u_n(x)&=\P_x(S_n=0)-\sum_{k=1}^n\P_x\Big(S_n=0,\ \varsigma_1=k\Big)\\
 &=\P_x(S_n=0)-\sum_{k=1}^n\left\{\P_x\Big(S_{k-1}=0,\ S_k\in \widetilde{B},\ S_n=0,\ \varsigma_1>k-1\Big)\right.\\
 &\qquad \qquad \qquad \qquad \qquad \qquad \left.+\P_x\Big(S_{k-1}\in \widetilde{B},\ S_k=0,\ S_n=0,\ \varsigma_1>k-1\Big)\right\}\\
 &=\P_x(S_n=0)-\sum_{k=1}^n\left\{\frac{|B|}{4}u_{k-1}(x)\P_{v_1}(S_{n-k}=0)+\frac{1}{4}v_{k-1}(x)\P_0(S_{n-k}=0)\right\}.
   \stepcounter{equation}\tag{\theequation} 
 \end{align*}
Similarly, we have
 \begin{align*}
 \label{lem3.1-eq2}
 v_n(x)&=\P_x(S_n\in \widetilde{B})-\sum_{k=1}^n \P_x\Big(S_n\in \widetilde{B},\ \varsigma_1=k\Big)\\
 &=\sum_{v\in \widetilde{B}}\P_{x-v}(S_n=0)-\sum_{k=1}^n \left\{\P_x\Big(S_{k-1}=0,\ S_k\in \widetilde{B},\ S_n\in \widetilde{B},\ \varsigma_1>k-1
\Big)\right.\\
 &\qquad \qquad \qquad \qquad \ \ \qquad \qquad \left.+\P_x\Big(S_{k-1}\in \widetilde{B},\ S_k=0,\ S_n\in \widetilde{B},\ \varsigma_1>k-1
\Big)\right\}\\
 &=\sum_{v\in \widetilde{B}}\P_{x-v}(S_n=0)\\
 &\qquad \qquad -\sum_{k=1}^n \left\{\frac{1}{4}u_{k-1}(x)\sum_{v,w\in \widetilde{B}}\P_{v-w}(S_{n-k}=0)+\frac{|B|}{4}v_{k-1}(x)\P_{v_1}(S_{n-k}=0)\right\}.
   \stepcounter{equation}\tag{\theequation} 
 \end{align*}
Recall that $G_x$ is the generating function of the sequence $(\P_x(S_n=0))_{n\ge0}$. Taking generating functions in (\ref{lem3.1-eq1}) and (\ref{lem3.1-eq2}) and using the convolution identity, we obtain
  \begin{align}
  \label{lem3.1-eq3}
   U_x(\xi)&=G_x(\xi)-\frac{|B|\xi}{4}U_x(\xi)G_{v_1}(\xi)-\frac{\xi}{4}V_x(\xi)G_0(\xi),\\
   \label{lem3.1-eq4}
   V_x(\xi)&=\Sigma_{x,B}(\xi)-\frac{\xi}{4}U_x(\xi)\Sigma_{B,B}(\xi)-\frac{|B|\xi}{4}V_{x}(\xi)G_{v_1}(\xi).
 \end{align}
Set $x=0$ in these two equations. Multiplying the former by $1+\frac{|B|\xi}{4}G_{v_1}(\xi)$ and subtracting $\frac{\xi}{4}G_0(\xi)$ times the latter gives
\begin{align*}
\label{DeltaU-G}
\Delta_B(\xi)U_0(\xi)=G_0(\xi).     \stepcounter{equation}\tag{\theequation} 
 \end{align*}
 Since $1\le U_0(\xi)\le G_0(\xi)<\infty$, we obtain $\Delta_B(\xi)>0$ for $0\le \xi<1$. Now, for $x\in\Z^2$, equations (\ref{lem3.1-eq3}) and (\ref{lem3.1-eq4}) can be rewritten in matrix form as
  \begin{align}
   \begin{pmatrix}
    \displaystyle 1+\frac{|B|\xi}{4}G_{v_1}(\xi) & \displaystyle \frac{\xi}{4}G_0(\xi) \\[10pt]
    \displaystyle \frac{\xi}{4}\Sigma_{B,B}(\xi) &  \displaystyle 1+\frac{|B|\xi}{4}G_{v_1}(\xi)
   \end{pmatrix}
   \begin{pmatrix}
     U_x(\xi)\\[10pt]
     V_x(\xi)
   \end{pmatrix}
   =
   \begin{pmatrix}
     G_x(\xi)\\[10pt]
     \Sigma_{x,B}(\xi)
   \end{pmatrix}.
 \end{align}
 Since the determinant of the matrix on the left-hand side is $\Delta_B(\xi)>0$, the matrix is invertible. Hence, the system can be solved, and the assertion follows.
\end{proof}

We next compute the generating function of $\varsigma_k$.

\begin{lem}
For $x\in \Z^2$, $k\ge 1$, and $0\le \xi< 1$, we have
\begin{align}
\label{lem3.2-maineq}
\P_x[\xi^{\varsigma_k}]=\begin{pmatrix}
  \displaystyle \frac{\xi}{4}V_x(\xi) & \displaystyle\frac{\xi}{4}U_x(\xi)
\end{pmatrix}M_B(\xi)^{k-1}
   \begin{pmatrix}
     1\\[10pt]
     |B|
   \end{pmatrix},
\end{align}
where
\begin{align}
  M_B(\xi):=\begin{pmatrix}
     \displaystyle \frac{\xi}{4}V_0(\xi) & \displaystyle\frac{\xi}{4}U_0(\xi)\\[10pt]
   \displaystyle  \frac{\xi}{4\Delta_B(\xi)}\Sigma_{B,B}(\xi) & \displaystyle\frac{\xi}{4}V_0(\xi) 
   \end{pmatrix}.
\end{align}
\end{lem}
\begin{proof}
Applying the Markov property at time $n-1$ to each summand, we obtain
\begin{align*}
\label{Lem3.2-eqA}
\P_x\Big[\xi^{\varsigma_1},\ S_{\varsigma_1}=0\Big]&=\sum_{n=1}^\infty \xi^n\P_x\Big(S_{n}=0,\ \varsigma_1=n\Big)\\
&=\frac{1}{4}\sum_{n=1}^\infty \xi^n\P_x\Big(S_{n-1}\in \widetilde{B},\ \varsigma_1>n-1\Big)\\
&=\frac{\xi}{4}V_x(\xi).
  \stepcounter{equation}\tag{\theequation} 
\end{align*}
For $v\in\widetilde{B}$, the same argument gives
\begin{align*}
\label{Lem3.2-eqB}
 \P_x\Big[\xi^{\varsigma_1},\ S_{\varsigma_1}=v\Big]&=\sum_{n=1}^\infty \xi^n\P_x\Big(S_{n}=v,\ \varsigma_1=n\Big)\\
&=\frac{1}{4}\sum_{n=1}^\infty \xi^n\P_x\Big(S_{n-1}=0,\ \varsigma_1>n-1\Big)\\
&=\frac{\xi}{4}U_x(\xi).
  \stepcounter{equation}\tag{\theequation} 
 \end{align*}
For $k\ge0$, the strong Markov property at $\varsigma_1$ gives
 \begin{align*}
 \label{lem3.2-eq1}
\P_x[\xi^{\varsigma_{k+1}}]&=\P_x\Big[\xi^{\varsigma_1}\cdot \P_{S_{\varsigma_1}}[\xi^{\varsigma_k}]\Big]\\
 &=\P_x\Big[\xi^{\varsigma_1},\ S_{\varsigma_1}=0\Big]\P_0[\xi^{\varsigma_k}]+\sum_{v\in \widetilde{B}}\P_x\Big[\xi^{\varsigma_1},\ S_{\varsigma_1}=v\Big]\P_{v}[\xi^{\varsigma_k}]\\
 &=\frac{\xi}{4}V_x(\xi)\P_0[\xi^{\varsigma_k}]+\frac{\xi}{4}U_x(\xi)\sum_{v\in \widetilde{B}}\P_{v}[\xi^{\varsigma_k}].
  \stepcounter{equation}\tag{\theequation}  
 \end{align*}
 In particular, when $x=0$, we have
 \begin{align*}
 \label{lem3.2-eq2}
 \P_0[\xi^{\varsigma_{k+1}}]=\frac{\xi}{4}V_0(\xi)\P_0[\xi^{\varsigma_k}]+\frac{\xi}{4}U_0(\xi)\sum_{v\in \widetilde{B}}\P_{v}[\xi^{\varsigma_k}].
  \stepcounter{equation}\tag{\theequation}  
 \end{align*}
 Substituting $x=v$ in (\ref{lem3.2-eq1}) and summing over $v\in \widetilde{B}$, {by applying Lemma \ref{lem:UandV}}, we obtain
 \begin{align*}
 \label{lem3.2-eq3}
 \sum_{v\in \widetilde{B}}\P_v[\xi^{\varsigma_{k+1}}]&=\frac{\xi}{4}\sum_{v\in \widetilde{B}}V_v(\xi)\P_0[\xi^{\varsigma_k}]+\frac{\xi}{4}\left(\sum_{v\in \widetilde{B}}U_v(\xi)\right)\left(\sum_{v\in \widetilde{B}}\P_{v}[\xi^{\varsigma_k}]\right)\\
 &=\frac{\xi}{4\Delta_B(\xi)}\Sigma_{B,B}(\xi)\P_0[\xi^{\varsigma_k}]+\frac{\xi}{4}V_0(\xi)\sum_{v\in \widetilde{B}}\P_{v}[\xi^{\varsigma_k}].
  \stepcounter{equation}\tag{\theequation}  
 \end{align*}
 Therefore, equations (\ref{lem3.2-eq2}) and (\ref{lem3.2-eq3}) yield the matrix recursion
 \begin{align}
   \begin{pmatrix}
     \P_0[\xi^{\varsigma_{k+1}}]\\[10pt]
     \displaystyle\sum_{v\in \widetilde{B}}\P_v[\xi^{\varsigma_{k+1}}]
   \end{pmatrix}=
   M_B(\xi)
   \begin{pmatrix}
     \P_0[\xi^{\varsigma_{k}}]\\[10pt]
     \displaystyle\sum_{v\in \widetilde{B}}\P_v[\xi^{\varsigma_{k}}]
   \end{pmatrix}.
 \end{align}
 Since $\varsigma_0=0$, iterating this yields
  \begin{align}
  \label{Lem3.2-eqC}
      \begin{pmatrix}
     \P_0[\xi^{\varsigma_{k}}]\\[10pt]
     \displaystyle\sum_{v\in \widetilde{B}}\P_v[\xi^{\varsigma_{k}}]
   \end{pmatrix}=
   M_B(\xi)^{k}
   \begin{pmatrix}
     1\\[10pt]
     |B|
   \end{pmatrix}.
 \end{align}
 Combining this identity with (\ref{lem3.2-eq1}), we obtain the desired result. This completes the proof.
\end{proof}

Next, we investigate the asymptotic behavior of $\Delta_B(q)$ as $q\to 1-$. Recall that $\kappa_B=\P_0(N_\tau^B\ge 1)$.

\begin{lem}
\label{lem3.3}
We have
\begin{align}
\label{lem3.3-eq1}
  \lim_{q\to 1-}\frac{\Delta_B(q)}{G_0(q)}=\frac{|B|}{2}-\frac{|B|^2}{8}+\frac{1}{16}\sum_{v,w\in \widetilde{B}}a(v-w)=\kappa_B.
\end{align}
\end{lem}
\begin{proof}
  By (\ref{GxFx-identity}), (\ref{a-definition}), and (\ref{avi=1}), we have
  \begin{align*}
 \frac{\Delta_B(q)}{G_0(q)}&=\frac{1}{G_0(q)}\left\{\left(1+\frac{|B|q}{4}G_{v_1}(q)\right)^2-\frac{q}{4}G_0(q)\left(\frac{q}{4}\sum_{v,w\in \widetilde{B}}G_{v-w}(q)\right)\right\}\\
 &=\frac{1}{G_0(q)}+\frac{|B|q}{2}F_{v_1}(q)+\frac{|B|^2q^2}{16}F_{v_1}(q)G_{v_1}(q)-\frac{q^2}{16}\sum_{v,w\in \widetilde{B}}G_{v-w}(q)\\
 &=\frac{1}{G_0(q)}+\frac{|B|q}{2}F_{v_1}(q)+\frac{|B|^2q^2}{16}F_{v_1}(q)\Big(G_{v_1}(q)-G_0(q)\Big)\\
 &\qquad +\frac{q^2}{16}\sum_{v,w\in \widetilde{B}}\Big(\Big(G_{v_1}(q)-G_0(q)\Big)+\Big(G_0(q)-G_{v-w}(q)\Big)\Big)\\
 &\to \frac{|B|}{2}-\frac{|B|^2}{16}a(v_1)+\frac{1}{16}\sum_{v,w\in \widetilde{B}}\Big(a(v-w)-a(v_1)\Big)\\
  &= \frac{|B|}{2}-\frac{|B|^2}{8}+\frac{1}{16}\sum_{v,w\in \widetilde{B}}a(v-w)
  \stepcounter{equation}\tag{\theequation}   
  \end{align*}
  as $q\to 1-$. This proves the first equality in (\ref{lem3.3-eq1}). 

To prove the second equality in (\ref{lem3.3-eq1}), we apply the strong Markov property at $\tau$ to obtain
\begin{align*}
U_0(q)&=\sum_{n=0}^\infty \P_0\Big(S_n=0,\ \varsigma_1>n\Big)q^n\\
&=1+\sum_{n=1}^\infty \sum_{k=1}^n\P_0\Big(S_n=0,\ \varsigma_1>n,\ \tau=k\Big)q^n\\
&=1+\sum_{n=1}^\infty \sum_{k=1}^n\P_0\Big(S_{n-k}=0,\ \varsigma_1>n-k\Big)\P_0\Big(\tau=k,\ k<\varsigma_1\Big)q^n\\
&=1+\sum_{k=1}^\infty \sum_{n=k}^\infty\P_0\Big(S_{n-k}=0,\ \varsigma_1>n-k\Big)\P_0\Big(\tau=k,\ k<\varsigma_1\Big)q^{n-k}q^k\\
&=1+\P_0\Big[q^\tau,\ \tau<\varsigma_1\Big]U_0(q).
\stepcounter{equation}\tag{\theequation}   
\end{align*}
Thus, by (\ref{DeltaU-G}), we have
\begin{align}
  \frac{\Delta_B(q)}{G_0(q)}=\frac{1}{U_0(q)}=1-\P_0\Big[q^\tau,\ \tau<\varsigma_1\Big]\to \P_0(\varsigma_1\le \tau)=\kappa_B
\end{align}
as $q\to 1-$. This completes the proof.
\end{proof}

Equation (\ref{lem3.3-eq1}) gives an explicit formula for $\kappa_B$. Its values for the possible configurations of $B$ are computed in Appendix \ref{AppendixB}.

Next, we determine the limits of $U_x(q)$ and $V_x(q)$ as $q\to 1-$.

\begin{lem}
\label{lem3.4}
For $x\in \Z^2$, we have
\begin{align}
\label{Def-Ux}
  \lim_{q\to 1-}U_x(q)&=\frac{1}{\kappa_B}\left(1-\frac{|B|}{4}-\frac{|B|}{4}a(x)+\frac{1}{4}\Sigma_{x,B}^a\right)=:U_x,\\
  \lim_{q\to 1-}V_x(q)&=\frac{1}{\kappa_B}\left(|B|-\frac{|B|^2}{4}+\frac{|B|^2}{4}a(x)-\frac{|B|}{4}\Sigma_{x,B}^a+\frac{1}{4}\Sigma_{B,B}^a\right)=:V_x,
\end{align}
where 
\begin{align}
 \Sigma_{x,B}^a:=\sum_{v\in \widetilde{B}}a(x-v),\qquad \Sigma_{B,B}^a:=\sum_{v,w\in \widetilde{B}}a(v-w).
\end{align}
\end{lem}
\begin{proof}
We first consider the case $x= 0$. Note that $a(0)=0$ and $\Sigma_{0,B}^a=|B|$ by (\ref{avi=1}). By (\ref{DeltaU-G}) and (\ref{lem3.3-eq1}), we have
\begin{align}
  \lim_{q\to 1-}U_0(q)=\lim_{q\to 1-}\frac{G_0(q)}{\Delta_B(q)}=\frac{1}{\kappa_B}=U_0.
\end{align}
By (\ref{Lem3.2-eqA}) and (\ref{Lem3.2-eqB}), we have
\begin{align}
  \P_0[q^{\varsigma_1}]=\frac{q}{4}\Big(V_0(q)+|B|U_0(q)\Big).
\end{align}
Since $\varsigma_1<\infty,\ \P_0$-a.s., we have $\P_0[q^{\varsigma_1}]\to 1$ as $q\to 1-$. Hence, we have
\begin{align}
  \lim_{q\to 1-}V_0(q)=4-|B|U_0=V_0.
\end{align}
This proves the assertions for $x=0.$

We next consider the case $x\neq 0$. By (\ref{GxFx-identity}), (\ref{lem3.3-eq1}), (\ref{a-definition}), and (\ref{avi=1}), we have
      \begin{align*}
U_x(q)&=\frac{G_0(q)}{\Delta_B(q)}\left(\left(1+\frac{|B|q}{4}G_{v_1}(q)\right)F_x(q)-\frac{q}{4}\Sigma_{x,B}(q)\right)\\
&=\frac{G_0(q)}{\Delta_B(q)}\left(F_x(q)+\frac{q}{4}\sum_{v\in \widetilde{B}}\Big(G_{v_1}(q)F_x(q)-G_{x-v}(q)\Big)\right)\\
&\to \frac{1}{\kappa_B}\left(1+\frac{1}{4}\sum_{v\in \widetilde{B}}\Big(-a(v_1)-a(x)+a(x-v)\Big)\right)\\
&=\frac{1}{\kappa_B}\left(1-\frac{|B|}{4}-\frac{|B|}{4}a(x)+\frac{1}{4}\Sigma_{x,B}^a\right)
  \stepcounter{equation}\tag{\theequation}   
  \end{align*}
and
  \begin{align*}
V_x(q)&=\frac{G_0(q)}{\Delta_B(q)}\left(\left(1+\frac{|B|q}{4}G_{v_1}(q)\right)\sum_{v\in \widetilde{B}}\frac{G_{x-v}(q)}{G_0(q)}-\frac{q}{4}F_x(q)\Sigma_{B,B}(q)\right)\\
&=\frac{G_0(q)}{\Delta_B(q)}\left(\sum_{v\in \widetilde{B}}\frac{G_{x-v}(q)}{G_0(q)}+\frac{q}{4}\sum_{v,w\in \widetilde{B}}\left(G_{v_1}(q)\cdot\frac{G_{x-v}(q)}{G_0(q)}-G_{v-w}(q)F_x(q)\right)\right)\\
&\to \frac{1}{\kappa_B}\left(\sum_{v\in \widetilde{B}}1+\frac{1}{4}\sum_{v,w\in \widetilde{B}}\Big(-a(v_1)+a(v-w)-a(x-v)+a(x)\Big)\right)\\
&=\frac{1}{\kappa_B}\left(|B|-\frac{|B|^2}{4}+\frac{|B|^2}{4}a(x)-\frac{|B|}{4}\Sigma_{x,B}^a+\frac{1}{4}\Sigma_{B,B}^a\right).
  \stepcounter{equation}\tag{\theequation}   
  \end{align*}
This completes the proof.
\end{proof}

For the explicit value of $\Sigma_{B,B}^a$, see Table \ref{table} {in Appendix \ref{AppendixB}}. By the argument in the above proof, the limits $U_x$ and $V_x$ satisfy
\begin{align}
\label{Vx+Ux}
  \frac{1}{4}\Big(V_x+|B|U_x\Big)=1.
\end{align}

The next lemma gives the limit of $M_B(q)$ as $q\to1-$, together with the eigenvalues and eigenvectors of the limiting matrix.

\begin{lem}
\label{lem3.5}
  We have
\begin{align}
   \lim_{q\to 1-}M_B(q)=\begin{pmatrix}
     \displaystyle 1-\frac{|B|}{4\kappa_B} & \displaystyle\frac{1}{4\kappa_B}\\[10pt]
   \displaystyle  \frac{|B|^2}{4\kappa_B} & \displaystyle 1-\frac{|B|}{4\kappa_B}
   \end{pmatrix}=:M_B.
\end{align}
Moreover, the eigenvalues of the matrix $M_B$ are 
\begin{align}
  \lambda^B_+:=1,\qquad \lambda^B_-:=1-\frac{|B|}{2\kappa_B}.
\end{align}
Corresponding eigenvectors can be chosen as follows:
\begin{align}
  \bm{e}_+:=\begin{pmatrix}
    1\\[10pt]
    |B|
  \end{pmatrix},\qquad 
  \bm{e}_-:=\begin{pmatrix}
    1\\[10pt]
    -|B|
  \end{pmatrix}.
\end{align}
\end{lem}
\begin{proof}
  The convergence in the first part follows from Lemmas \ref{lem3.3} and \ref{lem3.4}. The eigenvalues and eigenvectors are obtained by direct calculation.
\end{proof}

For the explicit value of $\lambda_B^-$, see Table \ref{table} {in Appendix \ref{AppendixB}}.

We define the following functions:
\begin{align}
\label{Def-H0}
  H_0^B(x)&:=\frac{1}{16\kappa_B}\Big(\left(\Sigma_{B,B}^a+4|B|-|B|^2\right)a(x)+(4-|B|)\Sigma_{x,B}^a+(4-|B|)^2\Big),\\
  \label{Def-Hk}
  H_k^B(x)&:=\frac{2}{|B|}-\frac{\Sigma_{B,B}^a}{128\kappa_B^2|B|}\left(1-\frac{|B|}{2\kappa_B}\right)^{k-1}\Big(\Sigma_{B,B}^a+2|B|^2a(x)-2|B|\Sigma_{x,B}^a\Big),
  \end{align}
  for $k\ge 1$. For $x\in\Z^2$ and $l\in\Z_{\ge 0}$, define
  \begin{align}
  \label{def-HfB}
  H_f^B(x,l)&:=\sum_{k=0}^\infty H_k^B(x)f(k+l).
\end{align}

We now prove (\ref{thm1.2-eq1}) for $l=0$; the general case follows by applying the same result to $f(\cdot +l)$.

\begin{prop}
\label{3.6}
For $x\in \Z^2$, we have
  \begin{align}
  \label{prop3.6-main}
   \lim_{q\to 1-}G_0(q)\P_x\Big[f(N_T^B)\Big]=H_f^B(x,0).
 \end{align}
\end{prop}
\begin{proof}
 As in the derivation of (\ref{lem2.1-eq1}), we have
  \begin{align*}
 \P_x(N_T^B\ge k)&=\P_x(\varsigma_k\le T)=\frac{1}{q}\P_x[q^{\varsigma_k}]\qquad \text{for}\ k\ge 1.
  \stepcounter{equation}\tag{\theequation}     
 \end{align*}
 Thus, we have
   \begin{align}
   \label{prop3.6-eq1}
    \P_x[f(N_T^B)]=\left(1-\frac{1}{q}\P_x[q^{\varsigma_1}]\right)f(0)+\frac{1}{q}\sum_{k\ge 1}\Big(\P_x[q^{\varsigma_{k}}]-\P_x[q^{\varsigma_{k+1}}]\Big)f(k).
  \end{align}
By (\ref{lem3.2-maineq}), (\ref{Uxq-def}), (\ref{Vxq-def}), (\ref{DeltaBq-def}), (\ref{lem3.3-eq1}), (\ref{a-definition}), (\ref{G_0(q)sim}), and (\ref{avi=1}), we have
\begin{align*}
\label{prop3.6-eq2}
&G_0(q)\left(1-\frac{1}{q}\P_x[q^{\varsigma_1}]\right)\\
&\qquad =\frac{G_0(q)}{4}\Big(4-V_x(q)-|B|U_x(q)\Big)\\
&\qquad =\frac{G_0(q)}{4\Delta_B(q)}\Bigg(4-2|B|(1-q)G_{v_1}(q)\\
&\qquad \qquad +\frac{q(1-q)}{4}\sum_{v,w\in \widetilde{B}}\Big(G_0(q)\Big(G_{v-w}(q)-G_{v_1}(q)\Big)+G_{v_1}(q)\Big(G_0(q)-G_{v_1}(q)\Big)\Big)\\
&\qquad \qquad +|B|\Big(G_{v_1}(q)-G_x(q)\Big)+\sum_{v\in \widetilde{B}}\Big(G_{v_1}(q)-G_{x-v}(q)\Big)\\
&\qquad \qquad + \frac{q}{4}\sum_{v,w\in \widetilde{B}}\Big\{\Big(G_{v_1}(q)-G_x(q)\Big)\Big(G_{v_1}(q)-G_{x-v}(q)\Big)\\
&\qquad \qquad \qquad \qquad \qquad \qquad \left.+\Big(G_{v-w}(q)-G_{x-v}(q)\Big)\Big(G_{x}(q)-G_{0}(q)\Big)\right\}\Bigg)\\
&\qquad \to \frac{1}{4\kappa_B}\left( 4+|B|\Big(a(x)-a(v_1)\Big)+\sum_{v\in \widetilde{B}}\Big(a(x-v)-a(v_1)\Big)\right.\\
&\qquad \qquad + \frac{1}{4}\sum_{v,w\in \widetilde{B}}\Big\{\Big(a(x)-a(v_1)\Big)\Big(a(x-v)-a(v_1)\Big)\\
&\qquad \qquad \qquad \qquad \qquad \qquad \left.+\Big(a(x-v)-a(v-w)\Big)\Big(-a(x)\Big)\right\}\Bigg)\\
&\qquad =H_0^B(x). 
 \stepcounter{equation}\tag{\theequation} 
\end{align*}

Next, we consider the second term in (\ref{prop3.6-eq1}). By (\ref{G_0(q)sim}) and (\ref{prop3.6-eq2}), we have
\begin{align}
  G_0(q)\Big(1-\P_x[q^{\varsigma_1}]\Big)=(1-q)G_0(q)+qG_0(q)\left(1-\frac{1}{q}\P_x[q^{\varsigma_1}]\right)\to H_0^B(x)
\end{align}
as $q\to 1-.$ Thus, by (\ref{Lem3.2-eqC}) {and (\ref{lem3.3-eq1})}, we have
\begin{align*}
G_0(q)\left\{\begin{pmatrix}
     1\\[10pt]
     |B|
   \end{pmatrix}-M_B(q)\begin{pmatrix}
     1\\[10pt]
     |B|
   \end{pmatrix}\right\}&=G_0(q)\begin{pmatrix}
     \displaystyle 1-\P_0[q^{\varsigma_1}]\\[10pt]
    \displaystyle \sum_{v\in \widetilde{B}}\Big(1-\P_v[q^{\varsigma_1}]\Big)
   \end{pmatrix}\\
   &\to \begin{pmatrix}
     H_0^B(0)\\[10pt]
     \displaystyle \sum_{v\in \widetilde{B}}H_0^B(v)
   \end{pmatrix}=\frac{2}{|B|}\bm{e}_+-\frac{\Sigma_{B,B}^a}{8\kappa_B|B|}\bm{e}_-.
  \stepcounter{equation}\tag{\theequation}   
  \end{align*}
For each fixed $k\ge1$, (\ref{lem3.2-maineq}), (\ref{Vx+Ux}), and Lemmas \ref{lem3.4} and \ref{lem3.5} yield
\begin{align*}
\label{prop3.6-eq3}
&G_0(q)\Big(\P_x[q^{\varsigma_{k}}]-\P_x[q^{\varsigma_{k+1}}]\Big)\\
&\qquad =G_0(q)\begin{pmatrix}
  \displaystyle \frac{q}{4}V_x(q) & \displaystyle\frac{q}{4}U_x(q)
\end{pmatrix}M_B(q)^{k-1}\left\{\begin{pmatrix}
     1\\[10pt]
     |B|
   \end{pmatrix}-M_B(q)\begin{pmatrix}
     1\\[10pt]
     |B|
   \end{pmatrix}\right\}\\
   &\qquad \to \begin{pmatrix}
 \displaystyle\frac{1}{4}V_x & \displaystyle\frac{1}{4}U_x
\end{pmatrix}(M_B)^{k-1}\left\{\frac{2}{|B|}\bm{e}_+-\frac{\Sigma_{B,B}^a}{8\kappa_B|B|}\bm{e}_-\right\}\\
&\qquad = \begin{pmatrix}
 \displaystyle\frac{1}{4}V_x & \displaystyle\frac{1}{4}U_x
\end{pmatrix}\left\{\frac{2}{|B|}\bm{e}_+-\frac{\Sigma_{B,B}^a}{8\kappa_B|B|}\left(1-\frac{|B|}{2\kappa_B}\right)^{k-1}\bm{e}_-\right\}\\
&\qquad =H_k^B(x).
  \stepcounter{equation}\tag{\theequation}   
  \end{align*}
  Therefore, since $f$ is summable, the assertion follows by applying the dominated convergence theorem to (\ref{prop3.6-eq1}) and using (\ref{prop3.6-eq2}) and (\ref{prop3.6-eq3}). This completes the proof.
\end{proof}

The martingale property and the penalization limit follow by the same arguments as in Lemma \ref{h-martingale-lem} and Proposition \ref{thm2.5}, respectively. We therefore omit the proof.

\begin{prop}
\label{thm3.7}
For every $x\in \Z^2$, the process 
\begin{align}
  \Big(H_f^B(S_n,N_n^B)\Big)_{n\ge 0}
\end{align}
is a nonnegative $\P_x$-martingale. Moreover, if $H_f^B(x,0)>0$, we have
 \begin{align}
 \label{thm3.7-eq1}
   \lim_{q\to 1-}\frac{\P_x[1_{\Lambda_n}\cdot f(N_{T}^{B})]}{\P_x[f(N_{T}^{B})]}=\P_x\left[1_{\Lambda_n}\cdot \frac{H_f^B(S_n,N_n^B)}{H_f^B(x,0)}\right]
 \end{align}
for $n\ge 0$ and $\Lambda_n\in \F_n$.
\end{prop}

Finally, we give necessary and sufficient conditions for $H_f^B$ to be strictly positive.

\begin{prop}
\label{lem3.8}
Let $l\in \Z_{\ge 0}$ {and $f : \Z_{\ge 0} \to [0, \infty)$}. If $|B|\le 3$, then we have
 \begin{align}
   H_f^B(x,l)>0\qquad \Longleftrightarrow\qquad \sum_{n=l}^\infty f(n)>0.
 \end{align}
If $|B|=4$, then we have
\begin{align}
  H_f^B(0,l)>0\qquad &\Longleftrightarrow\qquad {}^\exists n\ge l+1\ \text{s.t.}\ n-l\text{ is odd},\ f(n)>0,\\
   H_f^B(x,l)>0\qquad &\Longleftrightarrow\qquad {}^\exists n\ge l\ \text{s.t.}\ n-l\text{ is even},\ f(n)>0,
\end{align}
for $x\neq 0$.
\end{prop}
\begin{proof}
We first determine where $H_0^B$ is strictly positive. If $|B|<4$, then the definition of $H_0^B$ gives $H_0^B(x)>0$ for every $x\in\Z^2$. If $|B|=4$, that is, $B=B_\ast:=\{e_1,e_2,e_3,e_4\}$, then (\ref{a11}) and (\ref{lemA.1-eq2}) imply
\begin{align}
\label{lem3.8-eq1}
   \kappa_{B_\ast}=1,\qquad \Sigma_{B_\ast,B_\ast}^a=16.
 \end{align}
 Thus, we have
  \begin{align}
  \label{lem3.8-eq2}
   H_0^{B_\ast}(x)=\frac{1}{16\kappa_{B_\ast}}\Sigma_{B_\ast,B_\ast}^aa(x)=a(x).
 \end{align}
Therefore, we obtain $H_0^{{B_\ast}}(0)=0$, whereas $H_0^{{B_\ast}}(x)>0$ for $x\neq 0$.

Next, we consider $H_k^B$ for $k\ge 1.$ By direct calculation and (\ref{Vx+Ux}), we have
 \begin{align*}
 \label{lem3.8-eqA}
\left|\frac{1}{16\kappa_B}\Big(\Sigma_{B,B}^a+2|B|^2a(x)-2|B|\Sigma_{x,B}^a\Big)\right|&= \frac{1}{4}\Big|V_x-|B|U_x\Big|\le\frac{1}{4}\Big(V_x+|B|U_x\Big)=1.
  \stepcounter{equation}\tag{\theequation} 
 \end{align*}
 Thus, we have
\begin{align}
   H_k^B(x)\ge \frac{2}{|B|}-\frac{\Sigma_{B,B}^a}{8\kappa_B|B|}\cdot |\lambda^B_-|^{k-1}\qquad \text{for}\ x\in \Z^2.
 \end{align}
 If $|B|<4$, Table \ref{table} {in Appendix \ref{AppendixB}} gives $|\lambda^B_-|\le 1$. Therefore, we have
 \begin{align*}
   H_k^B(x)\ge \frac{2}{|B|}-\frac{\Sigma_{B,B}^a}{8\kappa_B|B|}=\frac{1}{|B|}\cdot \frac{|B|(4-|B|)}{4\kappa_B}>0\qquad \text{for}\ x\in \Z^2.
     \stepcounter{equation}\tag{\theequation}
 \end{align*}
 If $B=B_\ast$, Table \ref{table} {in Appendix \ref{AppendixB}} gives $\lambda^B_-=- 1$. By Proposition 4.4.2 of \cite{Lawler-Limic}, we have
  \begin{align}
   \Sigma_{x,B_\ast}^a=\begin{cases}
     4a(x)&\text{for}\ x\neq 0,\\[10pt]
     4&\text{for}\ x=0.
   \end{cases}
 \end{align}
 Therefore, by (\ref{lem3.8-eq1}), we obtain
 \begin{align}
 \label{lem3.8-eq3}
   H_k^{{B_\ast}}(0)&=\begin{cases}
     1&\text{if}\ k\ \text{is\ odd},\\
     0&\text{if}\ k\ \text{is\ even},
   \end{cases}\\
    \label{lem3.8-eq4}
   H_k^{B_\ast}(x)&=\begin{cases}
     0&\text{if}\ k\ \text{is\ odd},\\
     1&\text{if}\ k\ \text{is\ even {and $k \neq 0$}},
   \end{cases}\qquad \text{for}\ x\neq 0.
 \end{align}
The assertions now follow from (\ref{def-HfB}) and the positivity properties of $H_k^B$ established above. This completes the proof.
\end{proof}

This completes the proof of all the assertions in Theorem \ref{thm1.2}.

\subsection{Penalized Process}
\label{S3.2}
In this subsection, we derive path decomposition formulas for the penalized process. Most of the arguments parallel those in Subsection \ref{S2.2}, so we omit proofs that require only minor modifications.

For each $x\in\Z^2$ such that $H_f^B(x,0)>0$, let $\Q_x^{f,B\leadsto}$ denote the penalized measure on $\F_\infty$ arising in the limit (\ref{thm3.7-eq1}). As in the construction of (\ref{def-QfA}), this measure is defined by
\begin{align}
  \frac{d\Q_x^{f,B\leadsto}}{d\P_x}\Big|_{\F_n}:=\frac{H_f^B(S_n,N_n^B)}{H_f^B(x,0)}.
\end{align}
Under $\Q_x^{f,B\leadsto}$, the process $(S,N^B)$ is the $H_f^B$-transform of the original Markov chain $(S,N^B)$ and is therefore a strong Markov chain. In general, the coordinate process $S$ alone need not be Markovian. When $f$ is exponential, however, for example $f(k)=r^k$ with $0<r<1$, the process $S$ alone is Markovian.

Recall that $f^{(k)}(l):=1_{\{l=k\}}$. For each $k\ge 0$, define
\begin{align}
  Z_n^{[k],B}:=H_{f^{(k)}}^B(S_n,N_n^B)=H_{k-N_n^B}^B(S_n)1_{\{N_n^B\le k\}}.
\end{align}
The expression on the right-hand side is understood to be zero when $N_n^B>k$. By Proposition \ref{thm3.7}, the process $(Z_n^{[k],B})_{n\ge 0}$ is a nonnegative $\P_x$-martingale. In particular, we have
\begin{align}
  \P_x[Z_n^{[k],B}]=\P_x[Z_0^{[k],B}]=H_k^B(x).
\end{align}
The proof of Proposition \ref{lem3.8} shows that, when $|B|<4$, we have $H_k^B(x)>0$ for every $k\ge 0$ and $x\in\Z^2$. When $B=B_\ast$, however, $H_k^{B_\ast}(x)$ may vanish. For each $k\ge0$ and $x\in\Z^2$ such that $H_k^B(x)>0$, we can therefore define a probability measure $\Q_x^{[k],B\leadsto}$ on $\F_\infty$ by
\begin{align}
  \frac{d\Q_x^{[k],B\leadsto}}{d\P_x}\Big|_{\F_n}=\frac{Z_n^{[k],B}}{H_k^B(x)}.
\end{align}
In particular, we set $\P_x^{B\Updownarrow}:=\Q_x^{[0],B\leadsto}$, that is,
\begin{align}
\label{def-PxBupd}
   \frac{d\P_x^{B\Updownarrow}}{d\P_x}\Big|_{\F_n}:=\frac{H_0^B(S_n)}{H_0^B(x)}1_{\{N_n^B=0\}}\qquad \text{for}\ x\in \Z^2\ \text{with}\ H_0^B(x)>0.
\end{align}
The measures $\Q_x^{[k],B\leadsto}$ and $\P_x^{B\Updownarrow}$ do not depend on $f$. Under $\P_x^{B\Updownarrow}$, the coordinate process $S$ is the $H_0^B$-transform of the simple random walk killed at its first crossing of a $B$-edge. In particular, it is a strong Markov chain.

Arguing as in Lemmas \ref{lem-2.6} and \ref{lem2.7} and Corollary \ref{cor-2.8}, we obtain the following mixture representation and identify the distribution of the total number of $B$-edge crossings.

\begin{prop}
Assume that $|B|<4$. For $x\in \Z^2,$ we have
\begin{align}
\label{prop3.9-main}
  \Q_x^{f,B\leadsto}=\sum_{k=0}^\infty \frac{H_k^B(x)f(k)}{H_f^B(x,0)}\Q_x^{[k],B\leadsto}.
\end{align}
Let $N_\infty^B:=\lim_{n\to \infty}N_n^{B}$. Then
\begin{align}
\label{Prop3.9-eq1}
 \Q_x^{f,B\leadsto}(N_\infty^B=k)=\frac{H_k^B(x)f(k)}{H_f^B(x,0)}\qquad \text{for}\ k\ge 0.
\end{align}
In particular,
\begin{align}
  \Q_x^{f,B\leadsto}(N_\infty^B<\infty)=1.
\end{align}
\end{prop}

Suppose that $|B|<4$. Fix $k\ge1$ and $y\in\{0\}\cup\widetilde{B}$ such that the event $\{N_\infty^B=k,\ S_{\varsigma_k}=y\}$ has positive $\Q_x^{f,B\leadsto}$-probability. Conditional on this event, the path up to $\varsigma_k$ has the same law as the original simple random walk started at $x$, conditioned on $S_{\varsigma_k}=y$, and stopped at $\varsigma_k$. Under this conditional law, the post-$\varsigma_k$ process is independent of the stopped path and has law $\P_y^{B\Updownarrow}$. If the event $\{N_\infty^B=0\}$ has positive probability, then, conditional on this event, the entire path has law $\P_x^{B\Updownarrow}$. More precisely, we have the following proposition. Its proof is a straightforward adaptation of the proof of Proposition \ref{thm2.9} and is omitted.

\begin{prop}
\label{thm3.10}
Assume that $|B|<4$. For $x\in \Z^2$, $k\ge 1$, and $y\in \{0\}\cup \widetilde{B}$, we have
 \begin{align}
   \Q_x^{f,B\leadsto}\Big[F\cdot G\circ \theta_{\varsigma_k},\ N_\infty^B=k,\ S_{\varsigma_k}=y\Big]=\frac{H_0^B(y)f(k)}{H_f^B(x,0)}\cdot \P_x\Big[F,\ S_{\varsigma_k}=y\Big]\cdot \P_y^{B\Updownarrow}[G].
 \end{align}
We also have 
 \begin{align}
   \Q_x^{f,B\leadsto}\Big[G,\ N_\infty^B=0\Big]=\frac{H_0^B(x)f(0)}{H_f^B(x,0)}\P_x^{B\Updownarrow}[G],
 \end{align}
 where $F$ denotes a bounded $\F_{\varsigma_k}$-measurable random variable and $G$ denotes a bounded $\F_\infty$-measurable random variable.
\end{prop}

Suppose that $|B|<4$. Under $\P_0^{B\Updownarrow}$, the total number of completed excursions $L_\infty$ has a geometric distribution on $\Z_{\ge 0}$ with parameter $\kappa_B$. Conditional on $L_\infty=n$, the walk completes $n$ independent excursions, each having the law of an ordinary excursion conditioned not to cross any $B$-edge. It then leaves the origin along an edge chosen uniformly from $\{e_1,e_2,e_3,e_4\}\setminus B$, independently of these excursions. Conditional on the endpoint of this step being $v$, the subsequent path is independent of the preceding excursions and has law $\P_v^\updownarrow$.
For $x\ne 0$, the path decomposition depends on whether the walk hits the origin. Under $\P_x^{B\Updownarrow}$, conditional on $\tau<\infty$, the path up to $\tau$ has the same law as the original simple random walk started at $x$, conditioned on $N_\tau^B=0$, and stopped at $\tau$. Under this conditional law, the post-$\tau$ process is independent of the stopped path and has law $\P_0^{B\Updownarrow}$. Conditional on $\tau=\infty$, the entire path has law $\P_x^\updownarrow$. These statements are made precise in the following proposition. The proof follows by adapting the arguments used to prove Propositions \ref{thm2.10} and \ref{thm2.12} and is omitted.

\begin{prop}
\label{thm3.11}
Assume that $|B|<4$. For $n\ge 0$ and $v\in \{v_1,v_2,v_3,v_4\}\setminus \widetilde{B}$, we have
\begin{align*}
&\P_0^{B\Updownarrow}\left[\left(\prod_{k=1}^n G_k\circ \theta_{\rho_{k-1}}\right)\cdot H\circ \theta_{\rho_{n}+1},\ L_\infty =n,\ S_{\rho_n+1}=v \right]\\
   &\qquad =\frac{1}{4-|B|}\cdot \kappa_B\left(\prod_{i=1}^n \P_0\Big[G_i,\ N_\tau^B=0\Big]\right)\cdot \P_{v}^{\updownarrow}[H],
  \stepcounter{equation}\tag{\theequation} 
 \end{align*}
where $G_1,...,G_n$ denote bounded $\F_\tau$-measurable random variables and $H$ denotes a bounded $\F_\infty$-measurable random variable. For $n=0$, both products are interpreted as $1$. In particular,
 \begin{align}
   \P_0^{B\Updownarrow}(L_\infty=n)=\kappa_B(1-\kappa_B)^n\qquad \text{for}\ n\ge 0.
 \end{align}
Moreover, for $x\neq 0$, we have
\begin{align}
  \P_x^{B\Updownarrow}\Big[H,\ \tau=\infty\Big]&=\frac{a(x)}{H_0^B(x)}\P_x^{\updownarrow}[H],\\
  \P_x^{B\Updownarrow}\Big[G\cdot H\circ \theta_\tau,\ \tau<\infty\Big]&=\frac{H_0^B(0)}{H_0^B(x)}\P_x\Big[G,\ N_\tau^B=0\Big]\cdot \P_0^{B\Updownarrow}[H],
\end{align}
where $G$ denotes a bounded $\F_\tau$-measurable random variable and $H$ denotes a bounded $\F_\infty$-measurable random variable.
\end{prop}

We next relate $\P_0^{B\Updownarrow}$ to the measures introduced in Subsection \ref{S2.2}.

\begin{prop}
Assume that $|B|<4$, and let $A_B:=\{(i,j)\in\{1,2,3,4\}^2:\ v_i\in\widetilde{B}\text{ or }v_j\in\widetilde{B}\}$. Then
\begin{align*}
\P_0^{B\Updownarrow}=\P_{0,0}^{A_B\updownarrow}\Big(\cdot\mid S_{g+1}\notin \widetilde{B}\Big),
\end{align*}
where $g:=\sup\{n\ge 0:\ S_n=0\}$ denotes the last exit time from the origin.
\end{prop}
\begin{proof}
By Proposition \ref{thm2.10}, we have
\begin{align*}
   \P_{0,0}^{A_B\updownarrow}( S_{g+1}\notin \widetilde{B})&=\sum_{v\in \{v_1,v_2,v_3,v_4\}\setminus \widetilde{B}}\sum_{n=0}^\infty \P_{0,0}^{A_B\updownarrow}(S_{\rho_n+1}=v,\ L_\infty=n)\\
   &=\sum_{v\in \{v_1,v_2,v_3,v_4\}\setminus \widetilde{B}}\sum_{n=0}^\infty \frac{p_{A_B}}{4}(1-p_{A_B})^n\\
   &=\frac{4-|B|}{4}.
\stepcounter{equation}\tag{\theequation} 
\end{align*}
Since 
\begin{align}
  \Big\{(I_1,J_1)\in A_B\Big\}=\{N_\tau^B\ge 1\},
\end{align}
we have $p_{A_B}=\kappa_B.$ Thus, we have
\begin{align}
  \frac{\kappa_B}{4-|B|}=\frac{p_{A_B}}{4}\cdot \frac{1}{\P_{0,0}^{A_B\updownarrow}( S_{g+1}\notin \widetilde{B})}.
\end{align}
The assertion then follows by comparing Propositions \ref{thm2.10} and \ref{thm3.11} and applying the monotone class theorem. This completes the proof.
\end{proof}

Finally, we consider the case $B=B_\ast$. By (\ref{lem3.8-eq2}), we have $H_0^{B_\ast}(x)=a(x)$, and hence
\begin{align}
  \P_x^{B_\ast \Updownarrow}=\P_x^\updownarrow\qquad \text{for}\ x\neq 0.
\end{align}
Using this identity, we obtain the following path decomposition formula. The proof is a straightforward adaptation of the proof of Proposition \ref{thm2.9} and is omitted.

\begin{prop}
\label{3.13}
Let $x\in\Z^2$ satisfy $H_f^{B_\ast}(x,0)>0$. For $k\ge 1$ and $v\in\widetilde{B}_\ast$, we have
 \begin{align*}
  \Q_x^{f,B_\ast\leadsto}\Big[F\cdot G\circ \theta_{\varsigma_k},\ N_\infty^{B_\ast}=k,\ S_{\varsigma_k}=v\Big]&=\frac{f(k)}{H_f^{B_\ast}(x,0)}\P_x\Big[F,\ S_{\varsigma_k}=v\Big]\cdot \P_v^{\updownarrow}[G].
 \end{align*}
If $x\neq 0$, we also have
 \begin{align}
      \Q_x^{f,B_\ast\leadsto}\Big[G,\ N_\infty^{B_\ast}=0\Big]&=\frac{a(x)f(0)}{H_f^{B_\ast}(x,0)}\P_x^{\updownarrow}[G],
 \end{align}
 where $F$ denotes a bounded $\F_{\varsigma_k}$-measurable random variable and $G$ denotes a bounded $\F_\infty$-measurable random variable. Moreover, for every $k\ge 1$, we have
  \begin{align}
  \Q_x^{f,B_\ast\leadsto}\Big(N_\infty^{B_\ast}=k,\ S_{\varsigma_k}=0\Big)=0, 
 \end{align}
 and if $x=0$, then
 \begin{align}
   \Q_0^{f,B_\ast\leadsto}(N_\infty^{B_\ast}=0)=0.
 \end{align}
\end{prop}

This completes the path decomposition of the penalized process under $\Q_x^{f,B\leadsto}$.


\section{Universal $\sigma$-Finite Measure}
\label{S4}
In this section, we prove Theorem \ref{thm1.3}.

We first consider the penalized measure $\Q_x^{f,A\circlearrowleft}$. For each $x\in \Z^2$, define a measure $\mathscr{Q}_x^{A\circlearrowleft}$ on $\F_\infty$ by
\begin{align}
\mathscr{Q}_x^{A\circlearrowleft}&:=\Big(1+p_Aa(x)\Big)\Q_x^{(0),A\circlearrowleft}+ \sum_{k=1}^\infty \Q_x^{(k),A\circlearrowleft}.
\end{align}
By (\ref{lem2.7-eq1}), $\mathscr{Q}_x^{A\circlearrowleft}$ is $\sigma$-finite. We have
\begin{align*}
\label{universal-unweight}
1_{\{f(L_\infty^{A\circlearrowleft})>0\}}\mathscr{Q}_x^{A\circlearrowleft}&={h_f^A(x,0,0)}\cdot\sum_{k=0}^\infty \frac{w_f^{(k),A}(x)}{f(k)}\Q_x^{(k),A\circlearrowleft}=\frac{h_f^A(x,0,0)}{f(L_\infty^{A\circlearrowleft})}\cdot \Q_x^{f,A\circlearrowleft},
\stepcounter{equation}\tag{\theequation} 
  \end{align*}
where the sum is restricted to those $k\ge 0$ for which $f(k)>0$. Therefore, the measure $\mathscr{Q}_x^{A\circlearrowleft}$ may be regarded as the measure obtained from $\Q_x^{f,A\circlearrowleft}$ by removing the weight.

We now recall the results on local-time penalization obtained by Najnudel--Roynette--Yor \cite{Najnudel-Roynette-Yor}. Using the local time $\widetilde{L}_n:=\sum_{m=0}^n 1_{\{S_m=0\}}$ at the origin, they studied the penalization problem with weight $\Gamma_n=r^{\widetilde{L}_{n-1}}$ for $r\in (0,1)$. They showed that the resulting penalized measure $\widetilde{\Q}_x$ is given by
\begin{align}
  \frac{d\widetilde{\Q}_x}{d\P_x}\Big|_{\F_n}&:=\frac{a(S_n)+\frac{r}{1-r}}{a(x)+\frac{r}{1-r}}\cdot r^{\widetilde{L}_{n-1}}\qquad \text{for}\ n\ge 0,\ x\in \Z^2,
\end{align}
where we set $\widetilde{L}_{-1}:=0.$ Furthermore, they defined the measure $\widetilde{\mathscr{Q}}_x$, obtained by removing the weight, as follows:
\begin{align}
  \label{NRV-universal}
  \widetilde{\mathscr{Q}}_x&:=\left(a(x)+\frac{r}{1-r}\right)r^{-\widetilde{L}_\infty}\cdot \widetilde{\Q}_x\qquad \text{on}\ \F_\infty.
\end{align}
Theorem 4.1.3 of \cite{Najnudel-Roynette-Yor} shows that the measure $\widetilde{\mathscr{Q}}_x$ is independent of $r$ and is $\sigma$-finite. Consequently, the penalized measures corresponding to different values of $r$ can all be expressed in terms of the same measure $\widetilde{\mathscr{Q}}_x$. In this sense, $\widetilde{\mathscr{Q}}_x$ is universal.

We now show that $\widetilde{\mathscr{Q}}_x$ also serves as a universal measure for penalization by the weight $\Gamma_n=f(L_n^{A\circlearrowleft})$.

\begin{prop}
\label{thm4.1}
For $x\in \Z^2$, we have
  \begin{align}
   \mathscr{Q}_x^{A\circlearrowleft}= p_A\widetilde{\mathscr{Q}}_x.
 \end{align}
\end{prop}
\begin{proof}
For $r\in (0,1)$, let $f_r(k):=r^k$. Since $\mathscr{Q}_x^{A\circlearrowleft}$ does not depend on $f$, it suffices to consider $f=f_r$. Let
\begin{align}
  \widehat{\Q}_x^{r,A\circlearrowleft}:=h_{f_r}^A(x,0,0)\Q_x^{f_r,A\circlearrowleft}.
\end{align}
By the definition of $\Q_x^{f_r,A\circlearrowleft}$, we have
\begin{align*}
\frac{d\widehat{\Q}_x^{r,A\circlearrowleft}}{d\P_x}\Big|_{\F_n}&=h_{f_r}^A(x,0,0)\cdot\frac{d{\Q}_x^{f_r,A\circlearrowleft}}{d\P_x}\Big|_{\F_n} =h_{f_r}^A(S_n,L_n^{A\circlearrowleft},V_n)=r^{L_n^{A\circlearrowleft}}h_{f_r}^A(S_n,0,V_n).
  \stepcounter{equation}\tag{\theequation} 
\end{align*}
Thus, we obtain
\begin{align}
  \label{thm4.1-eqA}
  \frac{d\widehat{\Q}_x^{r,A\circlearrowleft}}{d\widetilde{\Q}_x}\Big|_{\F_n}=r^{L_n^{A\circlearrowleft}-\widetilde{L}_{n-1}}\cdot \frac{h_{f_r}^A(S_n,0,V_n)}{a(S_n)+\frac{r}{1-r}}\cdot \left(a(x)+\frac{r}{1-r}\right).
\end{align}

By Proposition 4.2.3 of \cite{Najnudel-Roynette-Yor}, the process $S$ under $\widetilde{\Q}_x$ is transient. 
Moreover, the path decomposition established in Subsection \ref{S2.2}, together with the transience of $S$ under  $\P_x^{\updownarrow}$ (see, e.g., Proposition 1.1 of \cite{GPV}), implies that $S$ is also transient under $\Q_x^{f_r,A\circlearrowleft}$. Thus, we have
\begin{align}
\label{thm4.1-eq1}
  L_n^{A\circlearrowleft}\to L_\infty^{A\circlearrowleft}<\infty,\qquad \widetilde{L}_n\to \widetilde{L}_\infty<\infty,
\end{align}
$\widetilde{\Q}_x\text{-a.s.\ and\ } \Q_x^{f_r,A\circlearrowleft}\text{-a.s.}$
Since $a(x)$ tends to infinity as $|x|\to \infty$ (see, e.g., Theorem 4.4.4 of \cite{Lawler-Limic}), we have
\begin{align}
  \lim_{n\to \infty}a(S_n)=\infty\qquad \widetilde{\Q}_x\text{-a.s.\ and\ } \Q_x^{f_r,A\circlearrowleft}\text{-a.s.}
\end{align}
Since $|r^A|\le 1$, we have
\begin{align}
\label{thm4.1-eq2}
 \lim_{n\to \infty}\frac{h_{f_r}^A(S_n,0,V_n)}{a(S_n)+\frac{r}{1-r}}=\lim_{n\to \infty}\frac{p_Aa(S_n)-r^A(S_n,V_n)+\frac{1}{1-r}}{a(S_n)+\frac{r}{1-r}}=p_A,
\end{align}
$\widetilde{\Q}_x\text{-a.s.\ and\ } \Q_x^{f_r,A\circlearrowleft}\text{-a.s.}$ By (\ref{thm4.1-eqA}), (\ref{thm4.1-eq1}), and (\ref{thm4.1-eq2}), Theorem 4.3.5 of \cite{Durrett} yields
\begin{align}
  \widehat{\Q}_x^{r,A\circlearrowleft}=r^{L_\infty^{A\circlearrowleft}-\widetilde{L}_\infty}\cdot p_A\left(a(x)+\frac{r}{1-r}\right)\cdot \widetilde{\Q}_x\qquad \text{on}\ \F_\infty.
\end{align}
Therefore, (\ref{universal-unweight}) and (\ref{NRV-universal}) imply the assertion. This completes the proof.
\end{proof}

We next consider the case of $\Q_x^{f,B\leadsto}$. For $x\in \Z^2$, we define the measure $\mathscr{Q}_x^{B\leadsto}$ on $\F_\infty$ as 
\begin{align}
  \mathscr{Q}_x^{B\leadsto}:=\sum_{k\ge 0:\ H_k^B(x)>0} H_k^B(x)\Q_x^{[k],B\leadsto}\qquad \text{on}\ \F_\infty.
\end{align}
By (\ref{Prop3.9-eq1}), $\mathscr{Q}_x^{B\leadsto}$ is $\sigma$-finite. For every $x$ such that $H_f^B(x,0)>0$, the mixture representation yields
\begin{align}
  1_{\{f(N_\infty^B)>0\}}\mathscr{Q}_x^{B\leadsto}=\frac{H_f^B(x,0)}{f(N_\infty^B)}\cdot \Q_x^{f,B\leadsto}\qquad \text{on}\ \F_\infty.
\end{align}
Therefore, the measure  $\mathscr{Q}_x^{B\leadsto}$ may be regarded as the measure obtained from $\Q_x^{f,B\leadsto}$ by removing the weight.

We now show that $\widetilde{\mathscr{Q}}_x$ also serves as a universal measure for penalization by the weight $\Gamma_n=f(N_n^B)$.

\begin{prop}
  For $x\in \Z^2$, we have
  \begin{align}
\mathscr{Q}_x^{B\leadsto}=\widetilde{\mathscr{Q}}_x\qquad \text{on}\ \F_\infty.
  \end{align}
\end{prop}
\begin{proof}
Fix $r\in(0,1)$, and set $f_r(k):=r^k$.  By an argument analogous to that used in the proof of Proposition \ref{thm4.1}, it suffices to establish the following limit:
  \begin{align}
  \label{thm4.2-eq3}
    \lim_{|x|\to \infty}\frac{H_{f_r}^B(x,0)}{a(x)}=1.
  \end{align}
By (\ref{Def-H0}), (\ref{Def-Ux}), and (\ref{lem3.3-eq1}), we have
\begin{align}
  H_0^B(x)=a(x)+\frac{4-|B|}{4}\cdot U_x.
\end{align}
Since $U_x,V_x\ge 0$, (\ref{Vx+Ux}) yields $0\le U_x\le 4/|B|$. Thus, we have
\begin{align}
\label{thm4.2-eq1}
  0\le H_0^B(x)-a(x)=\frac{4-|B|}{4}\cdot U_x\le \frac{4-|B|}{|B|}\qquad \text{for}\ x\in \Z^2.
\end{align}
For $k\ge 1$, (\ref{lem3.8-eqA}), (\ref{lem3.3-eq1}), {and $|\lambda^B_-|\le 1$} give
\begin{align*}
\left|H_k^B(x)-\frac{2}{|B|}\right|&=\left|\frac{\Sigma_{B,B}^a}{128\kappa_B^2|B|}\left(1-\frac{|B|}{2\kappa_B}\right)^{k-1}\Big(\Sigma_{B,B}^a+2|B|^2a(x)-2|B|\Sigma_{x,B}^a\Big)\right|\\
&\le \frac{\Sigma_{B,B}^a}{8\kappa_B|B|}\\
&=\frac{16\kappa_B-2|B|(4-|B|)}{8\kappa_B|B|}\\
&\le \frac{2}{|B|}.
  \stepcounter{equation}\tag{\theequation}
\end{align*}
Thus, we have
\begin{align}
\label{thm4.2-eq2}
  0\le H_k^B(x)\le \frac{4}{|B|}\qquad \text{for}\ k\ge 1,\ x\in \Z^2.
\end{align}
Combining (\ref{thm4.2-eq1}) and (\ref{thm4.2-eq2}), we obtain 
\begin{align}
  0\le H_{f_r}^B(x,0)-a(x)&=H_0^B(x)-a(x)+\sum_{k=1}^\infty H_k^B(x)r^k\le \frac{4-|B|}{|B|}+\frac{4r}{|B|(1-r)}.
\end{align}
This establishes (\ref{thm4.2-eq3}) and completes the proof.
\end{proof}

This completes the proof of all the assertions in Theorem \ref{thm1.3}.


\appendix
\section{Penalization by Local Time at the Origin}
\label{AppendixA}
In this appendix, we study penalization with weights of the form $\Gamma=f(L)$, where $f$ satisfies (\ref{assumption-f}). For exponential weights $f(k)=r^k$, with $r\in(0,1)$, the corresponding penalized measures can be obtained by normalizing the finite measures constructed in Section 4.1 of Najnudel--Roynette--Yor \cite{Najnudel-Roynette-Yor}; see also Subsection 4.3.5 therein for the two-dimensional simple random walk. We present the results for completeness and compare them with penalization by $L^{A\circlearrowleft}$ when $A=\{1,2,3,4\}^2$ and $N^B$ when $B=\{e_1,e_2,e_3,e_4\}$. The proofs of the assertions in this appendix are straightforward adaptations of the arguments in Section \ref{S2} and are omitted.

We first give the distribution of the local time at the geometric time $T$.

\begin{lem}
For $x\in \Z^2$, we have
\begin{align}
  \P_x(L_T=0)&=1-\frac{F_x(q)}{q},\\
  \P_x(L_T=k)&=\frac{1-F_0(q)}{q}\cdot F_x(q)F_0(q)^{k-1}\qquad \text{for}\ k\ge 1.
\end{align}
\end{lem}

For $x\in \Z^2$ and $l\in \Z_{\ge 0}$, we define
\begin{align}
h_f(x,l):=\Big(a(x)+1_{\{x=0\}}\Big)f(l)+\sum_{k=1}^\infty f(k+l).
\end{align}
In particular, $h_f$ is nonnegative, and assumption (\ref{assumption-f}) implies that
\begin{align}
  h_f(x,0)=\Big(a(x)+1_{\{x=0\}}\Big)f(0)+\sum_{k=1}^\infty f(k)>0\qquad \text{for}\ x\in \Z^2.
\end{align}

The following theorem gives the penalization limit.

\begin{thm}
\label{thmA.2}
For $x\in \Z^2$ and $l\in \Z_{\ge 0}$, we have
\begin{align}
\label{thmA.2-main}
  \lim_{q\to 1-}G_0(q)\P_x\Big[f(L_T+l)\Big]=h_f(x,l).
\end{align}
Moreover, the process
\begin{align}
  \Big(h_f(S_n,L_n)\Big)_{n\ge 0}
\end{align}
is a nonnegative $\P_x$-martingale. In addition, we have
 \begin{align}
 \label{Local-penal}
   \lim_{q\to 1-}\frac{\P_x[1_{\Lambda_n}\cdot f(L_{T})]}{\P_x[f(L_{T})]}=\P_x\left[1_{\Lambda_n}\cdot \frac{h_f(S_n,L_n)}{h_f(x,0)}\right],
 \end{align}
for every $n\ge 0,\ x\in \Z^2,$ and $\Lambda_n\in \F_n$.
\end{thm}

As in Subsection \ref{S2.2}, the martingale $h_f(S,L)$ defines a probability measure $\Q_x^f$ on $\F_\infty$ by
\begin{align}
\label{def-Qf}
  \frac{d\Q_x^f}{d\P_x}\Big|_{\F_n}=\frac{h_f(S_n,L_n)}{h_f(x,0)}.
\end{align}
We define
\begin{align}
  w_f^{\{0\}}(x):=\frac{(a(x)+1_{\{x=0\}})f(0)}{h_f(x,0)},\qquad w_f^{\{k\}}(x):=\frac{f(k)}{h_f(x,0)}\ \text{for}\ k\ge 1.
\end{align}
Set $\tau_0:=0$. For $k\ge 1$, let $\tau_k$ denote the time of the $k$-th visit to the origin at a strictly positive time:
\begin{align}
  \tau_1:=\tau,\qquad \tau_{k+1}:=\inf\Big\{n>\tau_k:\ S_n=0\Big\}.
\end{align}

Under $\Q_x^f$, the total local time $L_\infty:=\lim_{n\to\infty}L_n$ is almost surely finite. Let $k\ge 1$ satisfy $f(k)>0$. Conditional on $L_\infty=k$, the path up to $\tau_k$ has the same law as the original simple random walk started at $x$ and stopped at $\tau_k$. Under this conditional law, the post-$\tau_k$ process is independent of the stopped path. At time $\tau_k+1$, the walk moves from the origin to a site chosen uniformly from $\{v_1,v_2,v_3,v_4\}$. Conditional on this site being $v_i$, the process from time $\tau_k+1$ onward has law $\P_{v_i}^{\updownarrow}$.
If $f(0)>0$, the conditional law given $L_\infty=0$ is described as follows. When $x\neq 0$, the entire path has law $\P_x^{\updownarrow}$. When $x=0$, the walk first moves to a site chosen uniformly from $\{v_1,v_2,v_3,v_4\}$. Conditional on this site being $v_i$, the process from time $1$ onward has law $\P_{v_i}^{\updownarrow}$. More precisely, we have
the following theorem.

\begin{thm}
\label{thmA.3}
For $k\ge 1,\ i\in \{1,2,3,4\}$, and $x\in \Z^2$, we have
 \begin{align}
   \Q_x^{f}\Big[G\cdot H\circ \theta_{\tau_{k}+1},\  L_\infty=k,\ S_{\tau_k+1}=v_i\Big]=\frac{w_f^{\{k\}}(x)}{4}\P_x[G]\cdot\P_{v_i}^\updownarrow[H],
 \end{align}
where $G$ denotes a bounded $\F_{\tau_k}$-measurable random variable and $H$ denotes a bounded $\F_\infty$-measurable random variable. For $x\neq 0$, we have
 \begin{align}
   \Q_x^f\Big[H,\  L_\infty=0\Big]&=w_f^{\{0\}}(x)\P_x^\updownarrow[H].
 \end{align}
 Moreover, for $i\in \{1,2,3,4\}$, we have
 \begin{align}
     \Q_0^f\Big[H\circ \theta_1,\  L_\infty=0,\ S_1=v_i\Big]&=\frac{w_f^{\{0\}}(0)}{4}\P_{v_i}^\updownarrow[H].
 \end{align}
Here, $H$ denotes a bounded $\F_\infty$-measurable random variable.
\end{thm}

We now compare local-time penalization with penalization by $L^{A\circlearrowleft}$ in the case $A=A_\ast:=\{1,2,3,4\}^2$. In this case, $L^{A_\ast\circlearrowleft}$ counts all completed excursions from the origin. Thus, for every $n\ge 0$, we have
\begin{align}
\label{LA=L}
  L_n^{A_\ast\circlearrowleft}= L_n\qquad \P_0\text{-a.s.},
  \end{align}
  whereas, for $x\neq 0$,
  \begin{align}
  \label{LA+1=L}
  L_n^{A_\ast\circlearrowleft}+1_{\{\tau\le n\}}= L_n\qquad \P_x\text{-a.s.}
\end{align}
Moreover, we have
\begin{align}
  p_{A_\ast}=1,\qquad r^{A_\ast}(x,v)=\begin{cases}
    1&\text{for}\ x\neq 0,\ v=v_i,\\[10pt]
    0&\text{for}\ v=0.
  \end{cases}
\end{align}
Thus, we obtain
\begin{align}
  h_f^{A_\ast}(x,l,v)=\begin{cases}
    h_f(x,l)&\text{for}\ x\neq 0,\ v=v_i\ \text{or}\ x=0,\ v=0,\\[10pt]
    h_f(x,l)+f(l)&\text{for}\ x\neq 0,\ v=0.
  \end{cases}
\end{align}
These identities yield the following comparison of the penalized measures.

\begin{prop}
We have
  \begin{align}
  \label{propA4-eq1}
  \Q_0^{f,A_\ast\circlearrowleft}=\Q_0^f.
\end{align}
Define $\widehat{f}:=(f(0),f(0),f(1),f(2),...)$. Fix $x\neq 0$. Then we have
\begin{align}
\label{propA4-eq2}
\Q_x^{f,A_\ast\circlearrowleft}=\Q_x^{\widehat{f}}.
\end{align}
Moreover, if $f(0)>0$, then
\begin{align}
\label{propA4-eq3}
  \Q_x^{f,A_\ast\circlearrowleft}\Big(\cdot\mid L_\infty^{A_\ast\circlearrowleft}=0\Big)&=\frac{a(x)}{1+a(x)}\Q_x^{\widehat{f}}\Big(\cdot\mid L_\infty=0\Big)+\frac{1}{1+a(x)}\Q_x^{\widehat{f}}\Big(\cdot\mid L_\infty=1\Big).
  \end{align} 
For every $k\ge1$ such that $f(k)>0$, we have
  \begin{align}
  \label{propA4-eq4}
  \Q_x^{f,A_\ast\circlearrowleft}\Big(\cdot\mid L_\infty^{A_\ast\circlearrowleft}=k\Big)&=\Q_x^{\widehat{f}}\Big(\cdot\mid L_\infty=k+1\Big).
\end{align}
\end{prop}
\begin{proof}
Equation (\ref{propA4-eq1}) follows directly from (\ref{LA=L}). 

Let $x\neq 0.$ By (\ref{LA+1=L}), we have
\begin{align}
  f(L_n^{A_\ast \circlearrowleft})=\widehat{f}(L_n)\qquad \P_x\text{-a.s.}
\end{align}
The same identity holds with $n$ replaced by $T$. Comparing the penalization limits in Theorems \ref{thm1.1} and \ref{thmA.2} therefore gives (\ref{propA4-eq2}).

To prove (\ref{propA4-eq3}), assume that $f(0)>0$. By Theorem \ref{thmA.3} applied to $\widehat{f}$, we have
\begin{align}
  \Q_x^{\widehat{f}}(L_\infty=0)&=\frac{a(x)f(0)}{h_{\widehat{f}}(x,0)},\\
  \Q_x^{\widehat{f}}(L_\infty=1)&=\frac{f(0)}{h_{\widehat{f}}(x,0)}.
\end{align}
By (\ref{cor2.8-eq1}), we have
\begin{align*}
  \Q_x^{f,A_\ast\circlearrowleft}(L_\infty^{A_\ast\circlearrowleft}=0)&=\frac{(1+a(x))f(0)}{h_f^{A_\ast}(x,0,0)}=\frac{(1+a(x))f(0)}{h_{\widehat{f}}(x,0)}.
  \stepcounter{equation}\tag{\theequation} 
\end{align*}
The identity
\begin{align}
  \{L_\infty^{A_\ast \circlearrowleft}=0\}=\{L_\infty=0\}\cup \{L_\infty=1\}
\end{align}
holds up to a $\Q_x^{f,A_\ast\circlearrowleft}$-null set. Thus, we have
\begin{align*}
  \Q_x^{f,A_\ast\circlearrowleft}\Big[H\mid L_\infty^{A_\ast\circlearrowleft}=0\Big]&=\frac{\Q_x^{f,A_\ast\circlearrowleft}[H,\  L_\infty^{A_\ast\circlearrowleft}=0]}{\Q_x^{f,A_\ast\circlearrowleft}(L_\infty^{A_\ast\circlearrowleft}=0)}\\
  &=\frac{h_{\widehat{f}}(x,0)}{(1+a(x))f(0)}\cdot \left(\Q_x^{\widehat{f}}[H,\  L_\infty=0]+\Q_x^{\widehat{f}}[H,\  L_\infty=1]\right)\\
  &=\frac{a(x)}{1+a(x)}\Q_x^{\widehat{f}}\Big[H\mid L_\infty=0\Big]+\frac{1}{1+a(x)}\Q_x^{\widehat{f}}\Big[H\mid L_\infty=1\Big]
    \stepcounter{equation}\tag{\theequation} 
\end{align*}
for every bounded $\F_\infty$-measurable random variable $H$. This proves (\ref{propA4-eq3}).

Finally, let $k\ge 1$ satisfy $f(k)>0$. Equation (\ref{propA4-eq4}) follows from the identity
\begin{align}
  \{L_\infty^{A_\ast \circlearrowleft}=k\}=\{L_\infty=k+1\}\qquad \text{for}\ k\ge 1,
\end{align}
which holds up to a $\Q_x^{f,A_\ast\circlearrowleft}$-null set. This completes the proof.
\end{proof}

We now compare local-time penalization with penalization by $N^B$ in the case $B=B_\ast:=\{e_1,e_2,e_3,e_4\}$. In this case, $N^{B_\ast}$ counts all crossings of the edges incident to the origin. Thus, for every $n\ge 0$, we have
\begin{align}
\label{NBast-L}
  N_n^{B_\ast}=2L_n+1_{\{S_n\neq 0\}}\qquad \P_0\text{-a.s.},
 \end{align}
 whereas, for $x\neq 0$,
 \begin{align}
 \label{NBast-L2}
   N_n^{B_\ast}=2L_n-1_{\{S_n= 0\}}\qquad \P_x\text{-a.s.}
 \end{align}
Moreover, by (\ref{lem3.8-eq2}), (\ref{lem3.8-eq3}), and (\ref{lem3.8-eq4}), we have
\begin{align}
\label{HfBast}
  H_f^{B_\ast}(x,l)=\begin{cases}
    \displaystyle \sum_{k=0}^\infty f(2k+1+l)& \text{for}\ x=0,\\[10pt]
    \displaystyle a(x)f(l)+\sum_{k=1}^\infty f(2k+l)&\text{for}\ x\neq 0.
  \end{cases}
\end{align}
We define 
  \begin{align}
  g_x(k):=\begin{cases}
    f(2k+1)&\text{for}\ x=0,\\[10pt]
    f(2k)&\text{for}\ x\neq 0.
  \end{cases}
\end{align}

These identities yield the following comparison of the penalized measures.

\begin{prop}
Let $x\in\Z^2$ satisfy $H_f^{B_\ast}(x,0)>0$. We have
  \begin{align}
  \label{PropA5-eq1}
\Q_x^{f,B_\ast\leadsto}=\Q_x^{g_x}.
\end{align}
Moreover, if $x=0$, then for every $k\ge0$ such that $f(2k+1)>0$, we have
\begin{align}
  \label{PropA5-eq2}
  \Q_0^{f,B_\ast\leadsto}\Big(\cdot\mid N_\infty^{B_\ast}=2k+1\Big)=\Q_0^{g_0}\Big(\cdot\mid L_\infty=k\Big).
\end{align}
If $x\neq 0$, then for every $k\ge0$ such that $f(2k)>0$, we have
\begin{align}
  \label{PropA5-eq3}
  \Q_x^{f,B_\ast\leadsto}\Big(\cdot\mid N_\infty^{B_\ast}=2k\Big)=\Q_x^{g_x}\Big(\cdot \mid L_\infty =k\Big).
\end{align}
\end{prop}
\begin{proof}
  By (\ref{NBast-L}), (\ref{NBast-L2}), and (\ref{HfBast}), we have
  \begin{align}
    H_f^{B_\ast}(S_n,N_n^{B_\ast})&=h_{g_x}(S_n,L_n)\qquad \P_x\text{-a.s.},\\
    H_f^{B_\ast}(x,0)&=h_{g_x}(x,0).
  \end{align}
  This proves (\ref{PropA5-eq1}).

Next, assume that $x=0$. By Theorem \ref{thmA.3}, we have $L_\infty<\infty$,\ $\Q_0^{g_0}$-a.s. Thus, we have
\begin{align}
  1_{\{S_n\neq 0\}}\to 1,\qquad \Q_x^{g_0}\text{-a.s.}
\end{align}
Equation (\ref{NBast-L}) also holds under $\Q_0^{g_0}$, and letting $n\to\infty$, we obtain
\begin{align}
  N_\infty^{B_\ast}=2L_\infty +1,\qquad \Q_0^{g_0}\text{-a.s.}
\end{align}
Combining (\ref{PropA5-eq1}) with the above observation, we obtain (\ref{PropA5-eq2}).

Finally, (\ref{PropA5-eq3}) follows in the same manner. This completes the proof.
\end{proof}


\section{Penalization with a Constant Clock}
\label{AppendixC}
In this appendix, we show that the penalization limits for the constant and geometric clocks coincide when $f$ is nonincreasing or has finite support.

We first establish the following lemma.

\begin{lem}
\label{lemC.1}
Fix $x\in \Z^2$. Let $(F_m)_{m\ge 0}$ be a sequence of nonnegative, $\F_\infty$-measurable random variables that are integrable under $\P_x$. Assume that $\P_x[F_{m+1}]\le \P_x[F_m]$ for all $m\ge 0$ and that there exists a constant $C_x\ge 0$ such that
  \begin{align}
    \lim_{q\to 1-}G_0(q)\P_x[F_T]=C_x.
  \end{align}
  Then we have
  \begin{align}
  \label{LemC.1-eq1}
    \lim_{m\to \infty}\frac{1}{\pi}\log m\cdot \P_x[F_m]=C_x.
  \end{align}
\end{lem}
\begin{proof}
We first consider the case $C_x>0$. We have
\begin{align*}
  \sum_{n=0}^\infty q^n \P_x[F_n]&=\P_x[F_0]+\frac{q}{1-q}\sum_{n=1}^\infty (1-q)q^{n-1}\P_x[F_n]\\
  &=\P_x[F_0]+\frac{q}{G_0(q)(1-q)}\cdot G_0(q)\P_x[F_T].
  \stepcounter{equation}\tag{\theequation}
\end{align*}
By (\ref{G-log(1-q)}), we have
\begin{align}
  \lim_{q\to 1-}(1-q)\log \left(\frac{1}{1-q}\right)\sum_{n=0}^\infty q^n \P_x[F_n]=\pi C_x.
\end{align}
Thus, the discrete Tauberian theorem (see, e.g., Proposition 12.5.3 of \cite{Lawler-Limic}) yields (\ref{LemC.1-eq1}). 

Next, suppose that $C_x=0$. Let $q_m:=e^{-\frac{1}{m}}$. We have
\begin{align*}
  \P_x[F_{T_{1-q_m}}]&\ge \sum_{k=1}^m (1-q_m)q_m^{k-1}\P_x[F_k]\\
  &\ge \P_x[F_m](1-q_m)\sum_{k=1}^mq_m^{k-1}\\
  &=\Big(1-e^{-1}\Big)\P_x[F_m].
  \stepcounter{equation}\tag{\theequation}
\end{align*}
By (\ref{G-log(1-q)}), we have
\begin{align}
  G_0(q_m)\sim \frac{1}{\pi}\log \frac{1}{1-e^{-\frac{1}{m}}}\sim \frac{\log m}{\pi}\qquad \text{as}\ m\to \infty.
\end{align}
Thus, we have
\begin{align}
  0\le \frac{1}{\pi}\log m\cdot \P_x[F_m]\le \frac{\log m}{\pi G_0(q_m)}\cdot \frac{G_0(q_m)\P_x[F_{T_{1-q_m}}]}{1-e^{-1}}\to 0\qquad \text{as}\ m\to \infty.
\end{align}
This proves (\ref{LemC.1-eq1}) and completes the proof.
\end{proof}

When $f$ is nonincreasing, this lemma allows us to derive the corresponding constant-clock penalization limits from the geometric-clock limits obtained in the preceding sections.

\begin{thm}
Assume that $f$ satisfies (\ref{assumption-f}) and is nonincreasing; that is,
\begin{align}
  f(k)\ge f(k+1)\qquad \text{for}\ k\ge 0.
\end{align}
Then the penalization limits (\ref{penal-limit}), (\ref{penal-limit-2}), and (\ref{Local-penal}) also hold for the constant clock, that is,
\begin{align}
\label{PropC.2-eq1}
   \lim_{m\to \infty}\frac{\P_x[1_{\Lambda_n}\cdot f(L_{m}^{A\circlearrowleft})]}{\P_x[f(L_{m}^{A\circlearrowleft})]}&=\P_x\left[1_{\Lambda_n}\cdot \frac{h_f^A(S_n,L_n^{A\circlearrowleft},V_n)}{h_f^{A}(x,0,0)}\right],\\
   \label{PropC.2-eq2}
   \lim_{m\to \infty}\frac{\P_x[1_{\Lambda_n}\cdot f(N_{m}^{B})]}{\P_x[f(N_{m}^{B})]}&=\P_x\left[1_{\Lambda_n}\cdot \frac{H_f^B(S_n,N_n^B)}{H_f^B(x,0)}\right],\\
   \label{PropC.2-eq3}
 \lim_{m\to \infty}\frac{\P_x[1_{\Lambda_n}\cdot f(L_{m})]}{\P_x[f(L_{m})]}&=\P_x\left[1_{\Lambda_n}\cdot \frac{h_f(S_n,L_n)}{h_f(x,0)}\right],
 \end{align}
for every $n\ge 0$ and $\Lambda_n\in \F_n$. Equation (\ref{PropC.2-eq2}) holds for $x\in\Z^2$ such that $H_f^B(x,0)>0$, whereas the other two identities hold for all $x\in\Z^2$.
\end{thm}
\begin{proof}
By (\ref{lem2.2-eq3}), the sequence $F_m:=p_Af(L_m^{A\circlearrowleft})$, $m\ge 0$, satisfies the assumptions of {Lemma} \ref{lemC.1}. Hence,
\begin{align}
\label{propC2-exA}
  \lim_{m\to \infty}\frac{p_A}{\pi}\log m\cdot \P_x\Big[f(L_m^{A\circlearrowleft})\Big]=h_f^A(x,0,0).
\end{align}
Similarly, by (\ref{thm2.5-eq2}), the sequence $F_m=p_A1_{\Lambda_n}f(L_m^{A\circlearrowleft}),\ m\ge 0$, also satisfies the assumptions of Lemma \ref{lemC.1}. Hence,
\begin{align}
\label{propC2-exB}
  \lim_{m\to \infty}\frac{p_A}{\pi}\log m\cdot \P_x\Big[1_{\Lambda_n}\cdot f(L_{m}^{A\circlearrowleft})\Big]=\P_x\Big[1_{\Lambda_n}\cdot h_f^A(S_n,L_n^{A\circlearrowleft},V_n)\Big].
\end{align}
Combining (\ref{propC2-exA}) and (\ref{propC2-exB}), we obtain (\ref{PropC.2-eq1}). The limits (\ref{PropC.2-eq2}) and (\ref{PropC.2-eq3}) follow by the same argument.
\end{proof}

The same conclusion holds when $f$ has finite support, without any monotonicity assumption.

\begin{cor}
Assume that $f$ satisfies (\ref{assumption-f}) and has finite support; that is, there exists $M\in\Z_{\ge 0}$ such that $f(k)=0$ for all $k>M$. Then the penalization limits (\ref{PropC.2-eq1}), (\ref{PropC.2-eq2}), and (\ref{PropC.2-eq3}) hold.
\end{cor}
\begin{proof}
 For $j\ge 0$, let $g_j(k):=1_{\{k\le j\}}$. Then we have
 \begin{align}
 \label{corC.3-eq1}
   f(k)=\sum_{j=0}^M \Big(f(j)-f(j+1)\Big)g_j(k)\qquad \text{for}\ k\ge 0.
 \end{align}
Since each $g_j$ is nonincreasing, (\ref{propC2-exA}) holds with $f$ replaced by $g_j$. Using (\ref{corC.3-eq1}) and the linearity of the map $f\mapsto h_f^A$, we obtain
\begin{align*}
\lim_{m\to \infty}\frac{p_A}{\pi}\log m\cdot \P_x\Big[f(L_{m}^{A\circlearrowleft})\Big]&=\sum_{j=0}^M \Big(f(j)-f(j+1)\Big)\lim_{m\to \infty}\frac{p_A}{\pi}\log m\cdot \P_x\Big[g_j(L_{m}^{A\circlearrowleft})\Big]\\
&=\sum_{j=0}^M \Big(f(j)-f(j+1)\Big)h_{g_j}^A(x,0,0)=h_f^A(x,0,0).
 \stepcounter{equation}\tag{\theequation}
\end{align*}
The remaining cases can be proved in the same manner. This completes the proof.
\end{proof}


\section{Explicit Values of the Constants}
\label{AppendixB}
In this appendix, we compute the explicit values of the constants appearing throughout the paper.

For $i,j\in \{1,2,3,4\}$, let $p_{i,j}$ denote the conditional probability that the walk returns to the origin from $v_j$, given that its first step is from the origin to $v_i$:
\begin{align}
p_{i,j}:=\P_0\Big(S_{\tau-1}=v_j\mid S_1=v_i\Big).
\end{align}
The following lemma gives the explicit values of $p_{i,j}$.

\begin{lem}
For $i,j\in \{1,2,3,4\}$, we have
\begin{align}
\label{lemA.1-eq1}
  p_{i,j}=\begin{cases}
    \displaystyle \frac{1}{2}&\text{if}\ v_i=v_j,\\[10pt]
    \displaystyle \frac{1}{2}-\frac{1}{\pi}&\text{if}\ v_i\cdot v_j=0,\\[10pt]
    \displaystyle \frac{2}{\pi}-\frac{1}{2}& \text{if}\ v_i=-v_j.
  \end{cases}
\end{align}
\end{lem}
\begin{proof}
By the Markov property at time $1$, we have
\begin{align*}
\label{pij-identity}
  p_{i,j}&=\P_{v_i}(S_{\tau-1}=v_j)=\frac{1}{4}\P_{v_i}\left[\sum_{n=0}^{\tau-1}1_{\{S_n=v_j\}}\right]=\frac{1}{4}\Big(a(v_j)-a(v_j-v_i)+a(v_i)\Big).
  \stepcounter{equation}\tag{\theequation}
\end{align*}
In the last equality, we used the representation of the Green function of the walk killed at the origin in terms of the potential kernel (see, e.g., Proposition 4.6.3 of \cite{Lawler-Limic}). It is known that
\begin{align}
  a(x)=\frac{1}{(2\pi)^2}\int_{[-\pi,\pi]^2}\frac{1-\cos(x_1\theta_1+x_2\theta_2)}{1-\frac{1}{2}(\cos\theta_1+\cos\theta_2)}\,d\theta_1d\theta_2\qquad \text{for}\ x=(x_1,x_2)\in \Z^2
\end{align}
(see, e.g., Proposition 4.4.3 of \cite{Lawler-Limic}). This representation shows that $a(0)=0$ and
\begin{align}
a((x_1,x_2))=a((-x_1,x_2))=a((x_1,-x_2))=a((-x_1,-x_2))=a((x_2,x_1)).
\end{align}
Thus, we have
\begin{align*}
\label{avi=1}
  a(v_i)&=\frac{1}{2}\Big(a(v_1)+a(v_2)\Big)\\
&=\frac{1}{2}\cdot\frac{1}{(2\pi)^2}\int_{[-\pi,\pi]^2}\frac{(1-\cos\theta_1)+(1-\cos\theta_2)}{1-\frac{1}{2}(\cos\theta_1+\cos\theta_2)}\,d\theta_1d\theta_2\\
&=\frac{1}{(2\pi)^2}\int_{[-\pi,\pi]^2}\,d\theta_1d\theta_2\\
&=1
\stepcounter{equation}\tag{\theequation}
\end{align*}
for $i=1,...,4$. As shown on p. 89 of \cite{Lawler-Limic},
\begin{align}
\label{a11}
a((1,1))=a((1,-1))=a((-1,1))=a((-1,-1))=\frac{4}{\pi}.
\end{align}
Moreover, by Proposition 4.4.2 of \cite{Lawler-Limic}, we have
\begin{align*}
\label{lemA.1-eq2}
  a((2,0))&=4a((1,0))-a((0,0))-a((1,1))-a((1,-1))=4-\frac{8}{\pi}.
  \stepcounter{equation}\tag{\theequation}
\end{align*}
Combining (\ref{pij-identity}) with (\ref{avi=1})-(\ref{lemA.1-eq2}), we obtain (\ref{lemA.1-eq1}). This completes the proof.
\end{proof}

Since
\begin{align*}
  p_A&=\P_0\Big((I_1,J_1)\in A\Big)=\sum_{(i,j)\in A}\P_0\Big(S_1=v_i,\ S_{\tau-1}=v_j\Big)=\frac{1}{4}\sum_{(i,j)\in A}p_{i,j},
\stepcounter{equation}\tag{\theequation}
\end{align*}
we can compute $p_A$ explicitly using (\ref{lemA.1-eq1}). 

Next, we determine the explicit value of $\kappa_B$. By (\ref{lem3.3-eq1}), (\ref{a11}), and (\ref{lemA.1-eq2}), the explicit values of $\Sigma_{B,B}^a$, $\kappa_B,$ and $\lambda^B_-$ can be computed for each configuration of the edges belonging to $B$. The results are summarized in the following table.

  \begin{table}[!h]
 \begin{center}
  \begin{tabular}{c|ccc} \Xhline{1.5pt}
    Edge configuration of $B$ & $\Sigma_{B,B}^a$ & $\kappa_B$ & $\lambda^B_-$ \\ \hline
   $1$ edge & $0$ & $3/8$ & $-1/3$ \\
   $2$ adjacent edges & $8/\pi$ & $1/2+1/(2\pi)$ & $(1-\pi)/(1+\pi)$ \\
   $2$ opposite edges & $8-16/\pi$ & $1-1/\pi$ & $1/(1-\pi)$\\
   $3$ edges & $8$ & $7/8$ & $-5/7$ \\
   $4$ edges & $16$ & $1$ & $-1$\\ \Xhline{1.5pt}
  \end{tabular}
     \caption{Explicit values of the constants related to $N^B$}
     \label{table}
 \end{center}
\end{table}


\section*{Acknowledgments}
The first author acknowledges support from JSPS KAKENHI Grant Numbers 23K03158 and 25K24482. The second author acknowledges support from JSPS KAKENHI Grant Number 26KJ1605.

\addcontentsline{toc}{section}{References}
\bibliographystyle{plain}

\end{document}